\documentclass[12pt]{article}
\usepackage{graphics,epsfig,graphicx,color}
\graphicspath{{/EPSF/}{./Paper/}{Figures/}}
\usepackage{subfig}
\newsubfloat{figure}
\usepackage{amssymb,latexsym}

\usepackage{setspace}

\usepackage{amsmath}

\usepackage{mathrsfs,amsfonts}

\newif\ifPDF
\ifx\pdfoutput\undefined
	\PDFfalse
\else
	\ifnum\pdfoutput > 0
		\PDFtrue
	\else
		\PDFfalse
	\fi
\fi

\ifPDF
	\usepackage{pdftricks}
	\begin{psinputs}
		\usepackage{pstricks}
		\usepackage{pstcol}
		\usepackage{pst-plot}
		\usepackage{pst-tree}
		\usepackage{pst-eps}
		\usepackage{multido}
		\usepackage{pst-node}
		\usepackage{pst-eps}
	\end{psinputs}
\else
	\usepackage{pstricks}
\fi

\ifPDF
	\usepackage[debug,pdftex,colorlinks=true, 
	linkcolor=blue, bookmarksopen=false,
	plainpages=false,pdfpagelabels]{hyperref}
\else
	\usepackage[dvips]{hyperref}
\fi

\newtheorem{theorem}{Theorem}[section]

\newenvironment{proof}{\noindent{\bf Proof:}}{
    \hspace*{\fill} $\square$\medskip }

\newcommand{\be}{\mathbf e}

\newcommand{\E}{\varepsilon} 
\newcommand{\T}{\tau}

\newcommand{\G}{{\mathcal G}}
\newcommand{\Q}{{g_{_G}}}

\newcommand{\eps}{\varepsilon}

\newcommand{\old}[1]{}
\renewcommand{\be}{\begin{equation}}
\newcommand{\ee}{\end{equation}}

\title{The phases of large networks with
edge and triangle constraints}

\author{
Richard Kenyon\thanks{Department of Mathematics, Brown University, Providence, RI 02912; rkenyon@math.brown.edu}
\and Charles Radin\thanks{Department of Mathematics, University of Texas, Austin, TX 78712; radin@math.utexas.edu}  
\and Kui Ren \thanks{Department of Mathematics, University of Texas, Austin, TX 78712; ren@math.utexas.edu} 
\and Lorenzo Sadun\thanks{Department of Mathematics, University of Texas, Austin, TX 78712; sadun@math.utexas.edu} 
}

\begin{document}
\maketitle

\begin{abstract}
\noindent
Based on numerical simulation and local stability analysis we describe the structure of the phase space of
the edge/triangle model of random graphs. We support 
simulation evidence with mathematical proof of continuity and discontinuity for many of the phase transitions.
All but one of the
many phase transitions in this model break some form of symmetry, and we use this
model to explore how changes in symmetry are related to discontinuities at these transitions. \end{abstract}



\section{Introduction}
\label{SEC:Intro}

We use the variational formalism of equilibrium statistical mechanics
to analyze large random graphs. More specifically we analyze ``emergent
phases", which represent the (nonrandom) large-scale structure of
typical large graphs under 
global constraints on subgraph densities. We concentrate on
the model with edge and triangle density constraints, $\E$ and $\T$,
which in this context play somewhat similar roles that mass and energy density constraints play in
microcanonical models of simple materials. Our goal is to understand the statistical states
$g_{\E,\T}$ which maximize entropy for given $\E$ and $\T$. These are the
analogues of Gibbs states in statistical mechanics.

Other parametric models of random graphs are
widely used, in particular exponential random graph models (analogues
of grand canonical models of materials), and the edge/triangle
constraints have been studied in that formalism since
they were popularized by Strauss in 1986 \cite{St}. (There is a short discussion
of exponential models in Section \ref{SEC:notation}.)

This paper, following on
\cite{RS1,RS2,RRS1,KRRS1,KRRS2,RRS2}, is the first attempt to
determine the qualitative features of  the whole phase space of the
edge/triangle model. We discuss what phases exist, and give the basic features of the
transitions between these phases; see Figure \ref{phase-diagram}.

\begin{figure}[htbp]
\center{\includegraphics[width=5in]{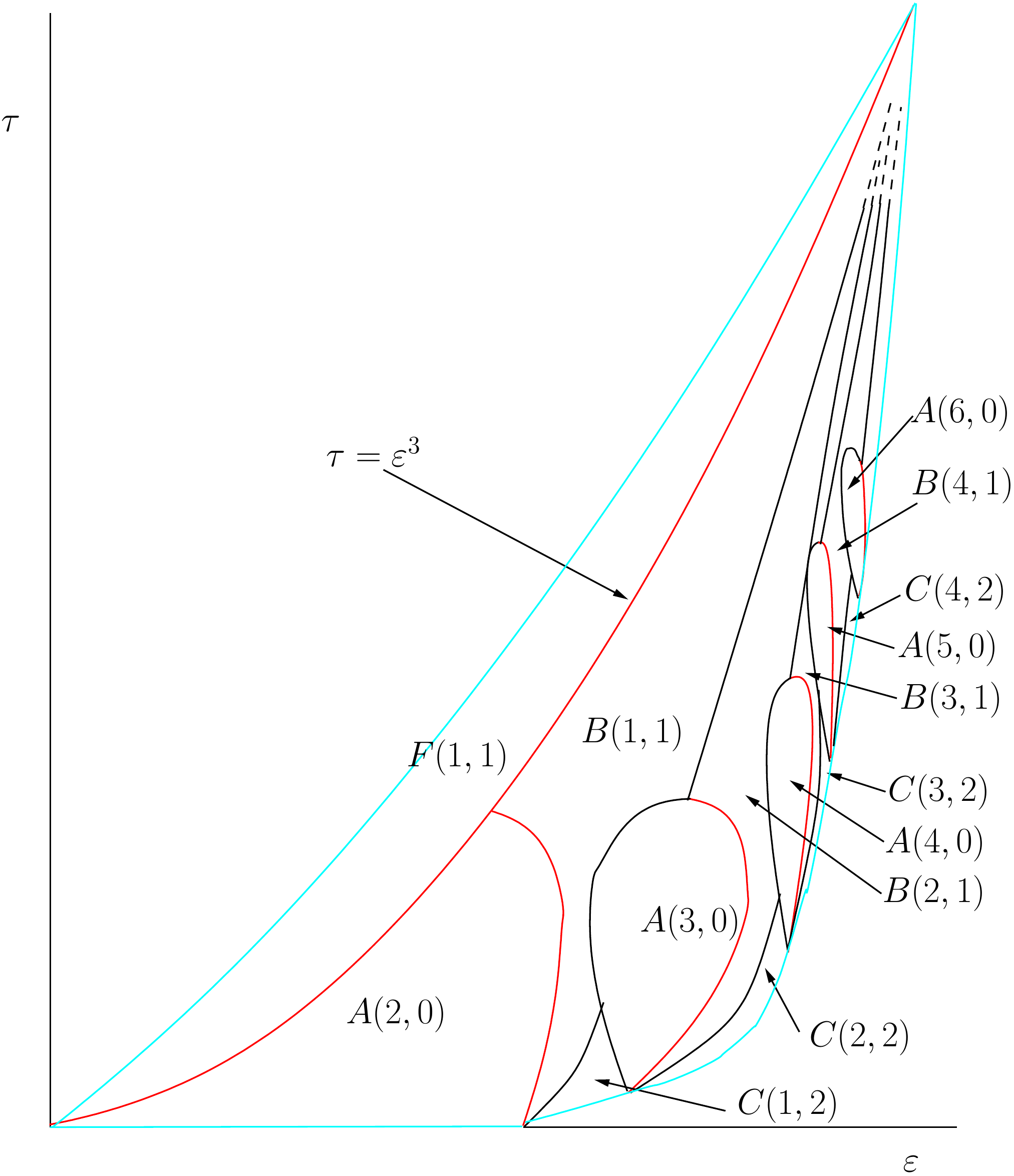}}
\caption{Distorted sketch of 14 of the phases in the edge/triangle model. 
Continuous phase transitions are shown in red, and discontinuous
phase transitions in black.}
\label{phase-diagram}
\end{figure}

The phase space $\Gamma$ is the space of achievable values of the
constraints, and a phase is a connected open subset of $\Gamma$
in which the entropy optimizing $g_{\E,\T}$ are unique for given $(\E,\T)$, and
vary analytically with
$(\E,\T)$.  In the case of edge/triangle constraints $\Gamma$ is the
``scalloped triangle" of Razborov \cite{R3} (see Figure~\ref{Razborov-triangle}).

Determining the optimal states in the edge/triangle model is a difficult
variational problem, one that has not been 
solved rigorously except in special 
regions of the phase space \cite{RS1,RS2,RRS1,KRRS1,KRRS2,RRS2}. However there is 
solid evidence that each individual optimizer
is \emph{multipodal}, that is, described by a stochastic block model (see definition below).
This evidence is borne out by our numerical studies, and based on them we conjecture that the optimal states
occur in three families (plus one extra phase), corresponding to certain 
equivalences or symmetries, as follows. We will give the evidence for
our conjectures as we proceed.

In the region $\T<\E^3$ of $\Gamma$ there are three infinite families of
phases, denoted $A(n,0)$, $B(n,1)$ and $C(n,2)$; See Figure~\ref{phase-diagram}.
Each statistical state in an $A(n,0)$ phase corresponds to a partition
of the set $\Sigma$ of nodes into $n\ge 2$ subsets $\Sigma_j$ which are 
{\it equivalent} in the sense that the sizes of all $\Sigma_j$ are $1/n$
of the whole,
the probability $p_{jk}$ of an edge between $v\in\Sigma_j$ and $w\in\Sigma_k$
is independent of $j$ and $k$, only depending on whether $j=k$ or $j\ne k$. 
It follows that for an $A(n,0)$ phase there are only two 
statistical parameters $p_{11}$ and $p_{12}$, which are thus easily computable functions of the edge and
triangle constraints, $\E$ and $\T$. See for example Figure \ref{phaseA} for the case $A(2,0)$,
and the left graphon in Figure \ref{ABCexample}.
\begin{figure}[htbp]
\center{\includegraphics[width=5in]{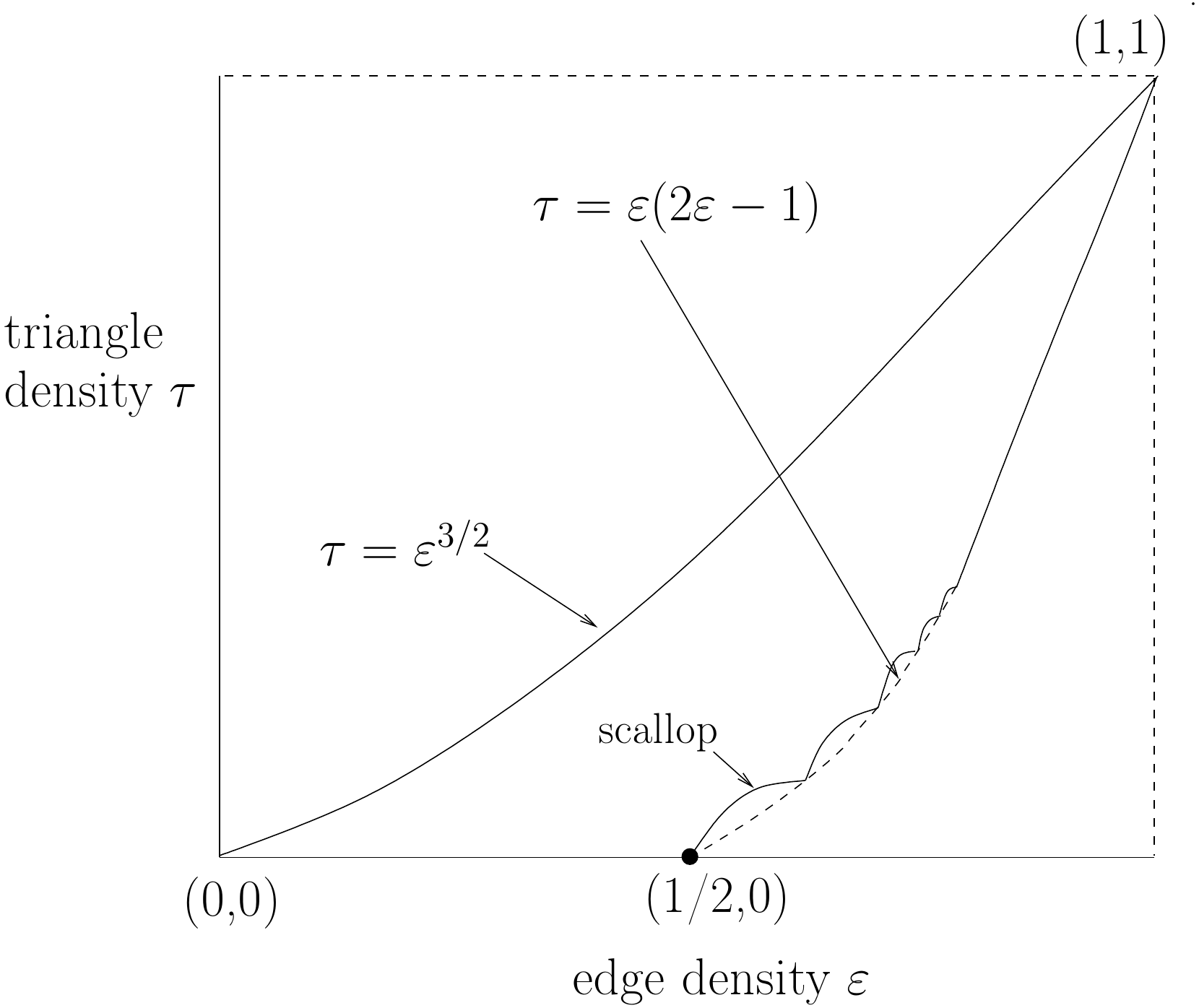}}
\caption{Distorted sketch of Razborov's scalloped triangle. The dotted curve is not part of the region; it just indicates
the quadratic curve on which cusps lie.}
\label{Razborov-triangle}
\end{figure}

The $A(n,0)$ phase touches
the lower boundary of $\Gamma$ at the cusp $(\E,\T)=
\left ( \frac{n}{n+1}, \frac{n(n-1)}{(n+1)^2} \right)$. 

A statistical state in a $B(n,1)$ phase corresponds to a partition of
$\Sigma$ into $n$ statistically equivalent subsets, $n\ge 1$, plus one
other set of statistically equivalent nodes, see Figure
\ref{ABCexample} middle. In the phase diagram, these phases are
arranged in stripes coming out of the point $(\E,\T)=(1,1)$, and part
of the boundary of $B(n,1)$ is shared with $A(n+1,0), A(n+2,0)$ and
with $C(n,2)$. These are depicted schematically in Figure~\ref{phase-diagram}.

A statistical state in a $C(n,2)$ phase corresponds to a nodal
partition into $n$ equivalent subsets, $n\ge 1$,
plus another statistically equivalent pair of subsets, see Figure \ref{ABCexample} right.
The phase
shares boundary with $B(n,1)$, $A(n+2,0)$, and the scallop
connecting the cusps with $\E=n/(n+1)$ and
$\E=(n+1)/(n+2)$. See Figure~\ref{phase-diagram}.

$F(1,1)$ denotes the single phase for the region 
$\T>\E^3$; see Figure~\ref{phase-diagram}. 
A statistical state in
$F(1,1)$ corresponds to a bipodal graphon.
Although this is the same basic structure as $B(1,1)$, the two phases
are different in more subtle ways.

\begin{figure}[htbp]
\center{\includegraphics[width=3in]{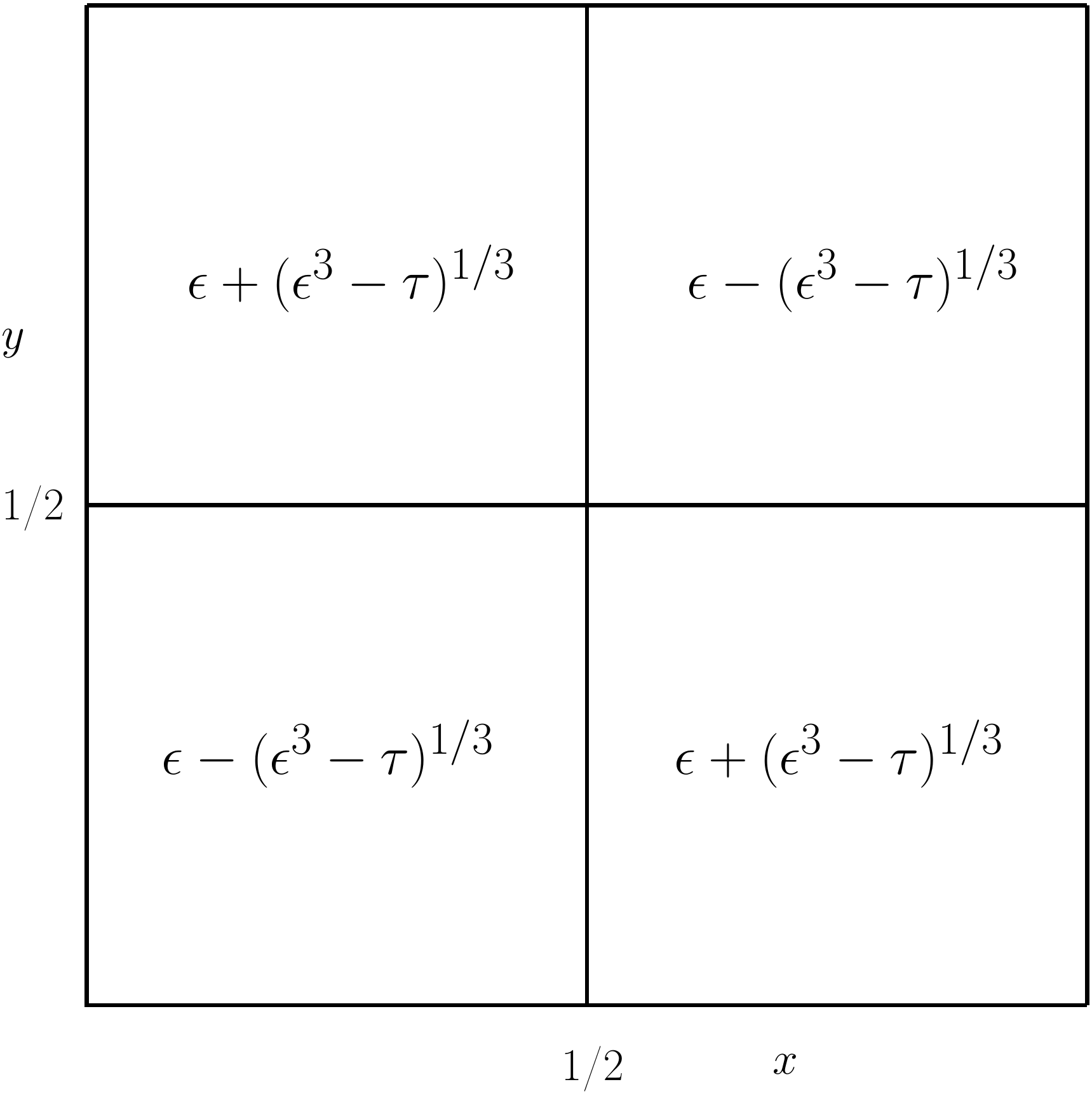}}
\caption{The optimal graphons in phase $A(2,0)$}
\label{phaseA}
\end{figure}

\begin{figure}[htbp]
\center{\includegraphics[width=2in]{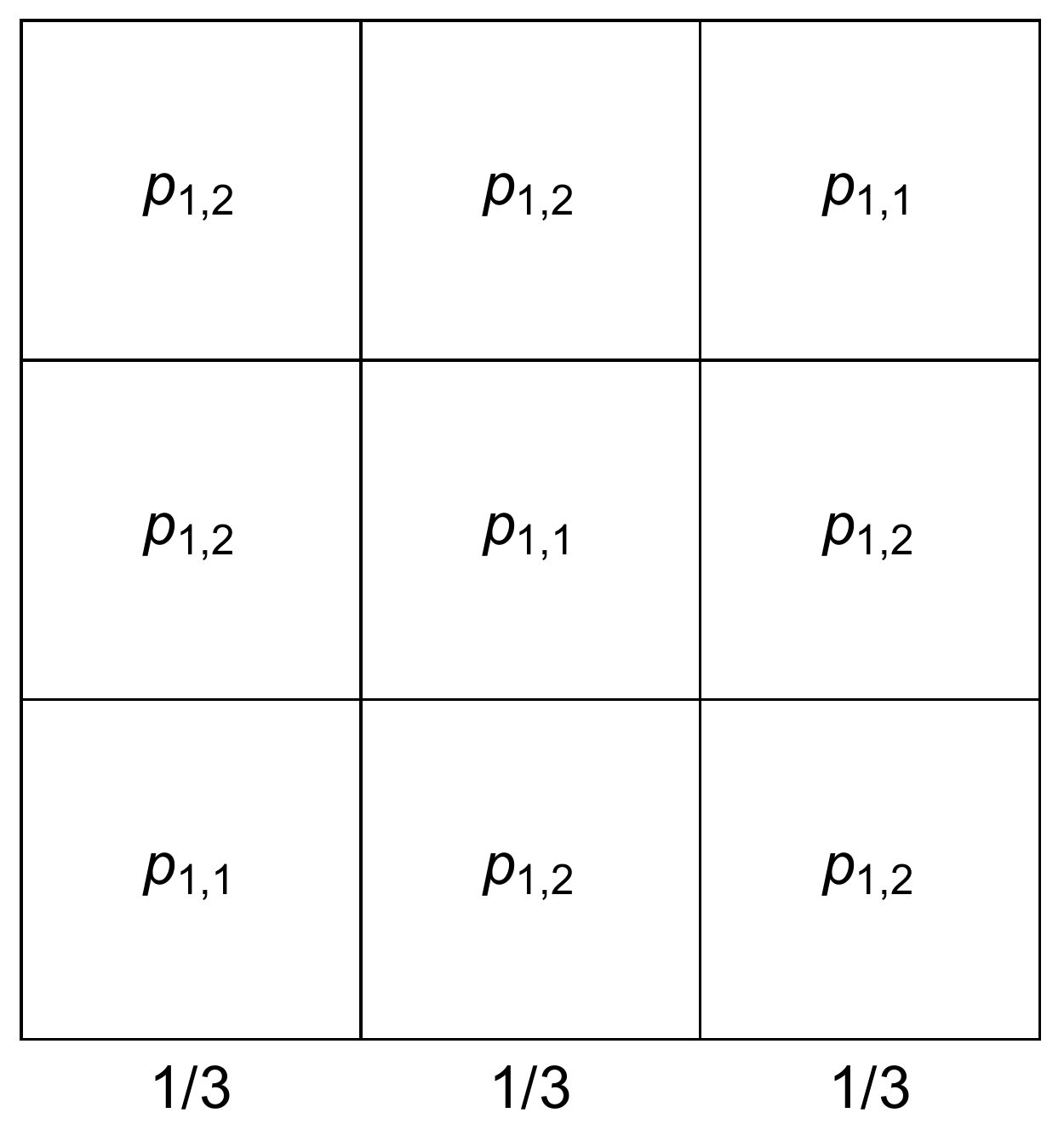}\includegraphics[width=2in]{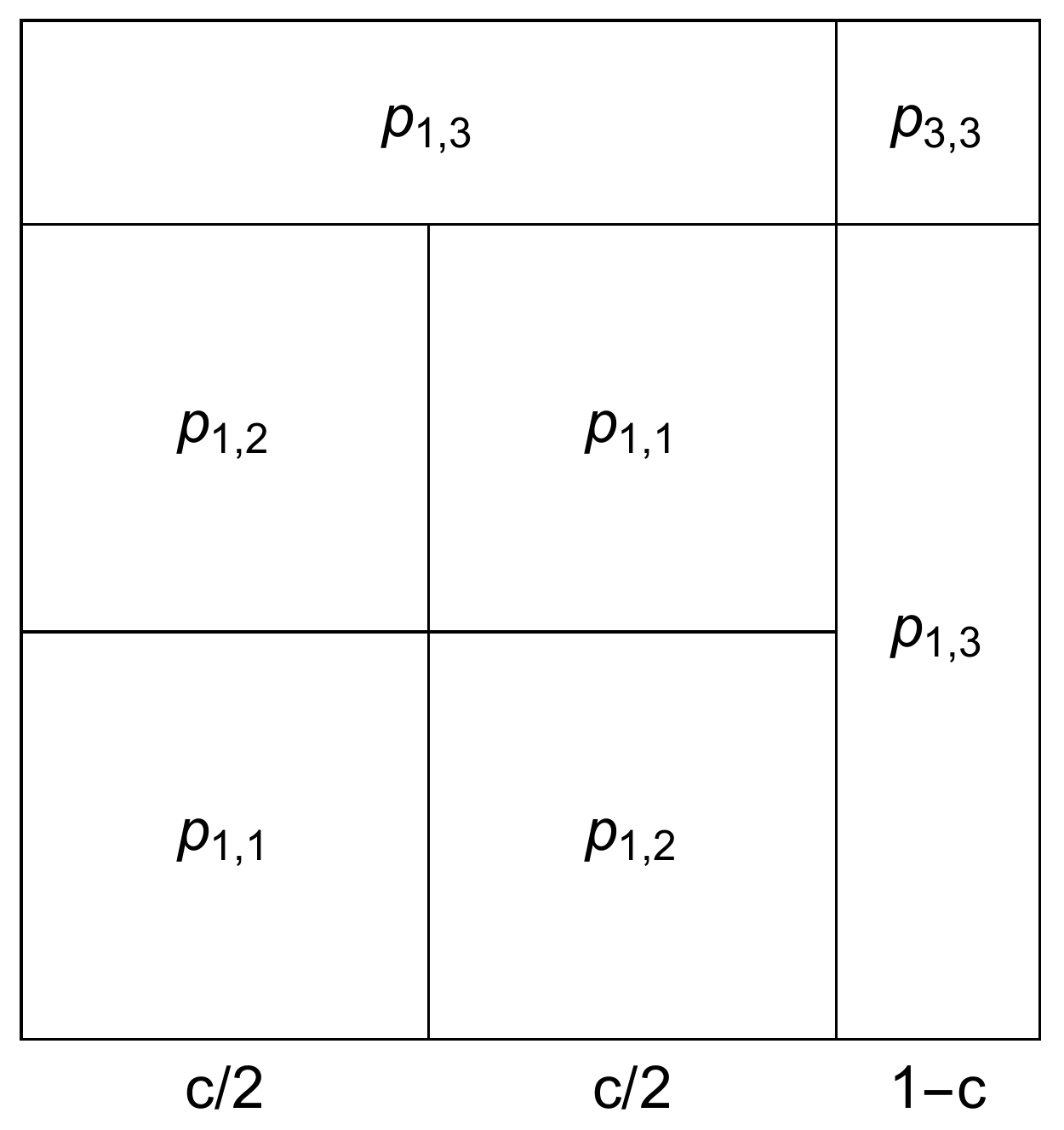}\includegraphics[width=2in]{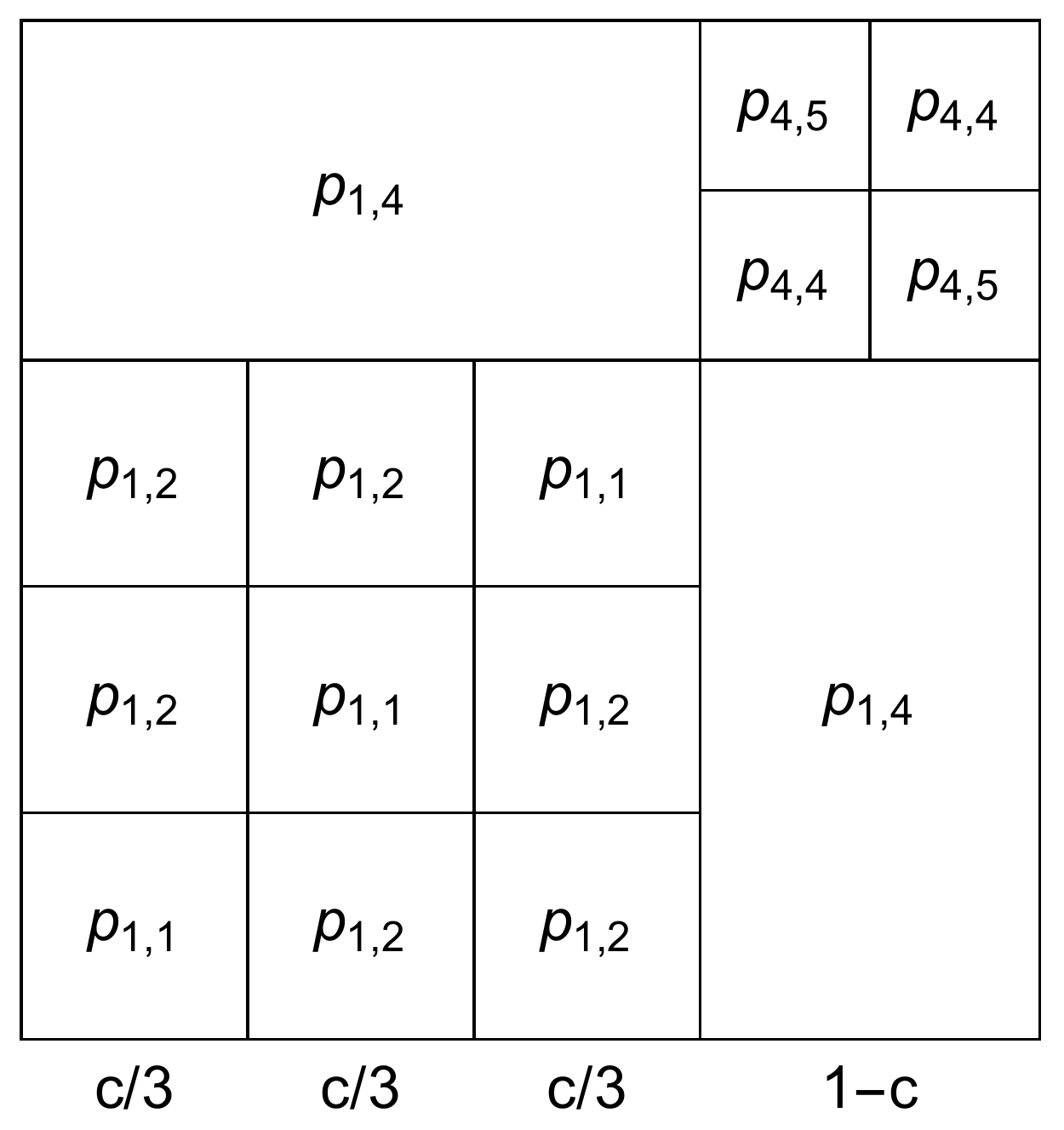}}
\caption{\label{ABCexample}Example of an $A(3,0)$, a  $B(2,1)$ and $C(3,2)$ graphon.}
\end{figure}
The multipodal structure of
the $g_{\E,\T}$ has only been proven on certain subsets of the phase space: on the boundary, on the
Erd\H{o}s-R\'{e}nyi curve $\T=\E^3$, on the line segment $(1/2,\T)$
for $0\le \T\le 1/8$, and in a region just above the
Erd\H{o}s-R\'{e}nyi curve; see \cite{KRRS2}. It has however also been proven
just above the
Erd\H{o}s-R\'{e}nyi curve in a wide family of other models \cite{KRRS2}, for
all parameters in edge/$k$-star models \cite{KRRS1}, and is supported in our
edge/triangle model by
extensive simulations for a range of constraints \cite{RRS1}.

The support for the specific pattern of phases in
Figure~\ref{phase-diagram} is purely from simulation and will be described
in Section~\ref{SEC:numerics}.

In Section~\ref{SEC:transitions} we use this detailed structure of the
optimal states to analyze the transitions between phases. In our
edge/triangle model all phase transitions involve a change of
symmetry. 
Some of the
transitions are continuous and some discontinuous, and we prove some
of these results in that section. The role of symmetry in determining
these qualitative features is of interest, given the major
role that symmetry plays in the understanding of phase transitions in
statistical physics~\cite{LL,An}. We discuss this further in Section~\ref{SEC:symmetry}.


\section{Notation and formalism}
\label{SEC:notation}

The study of large dense
graphs uses the mathematical tool of graphons, which we now review \cite{L}, making use of the
discussion in \cite{Ra}. We let $G_n$ denote the set of graphs on $n$ 
nodes, which we label $\{1,\ldots,n\}$. Graphs are assumed simple, i.e.\ undirected
and without
multiple edges or loops. A graph $G$ in $G_n$ can be represented by
its $0,1$-valued adjacency matrix, or alternatively by the function 
$\Q$ on the unit square $[0,1]^2$ with constant value 0 or 1 in each
of the subsquares of area $1/n^2$ centered at the points $((j-\frac12)/n,
(k-\frac12)/n)$; see Figure~\ref{multipartite}. 

\begin{figure}[htbp]
\center{\includegraphics[width=2.5in]{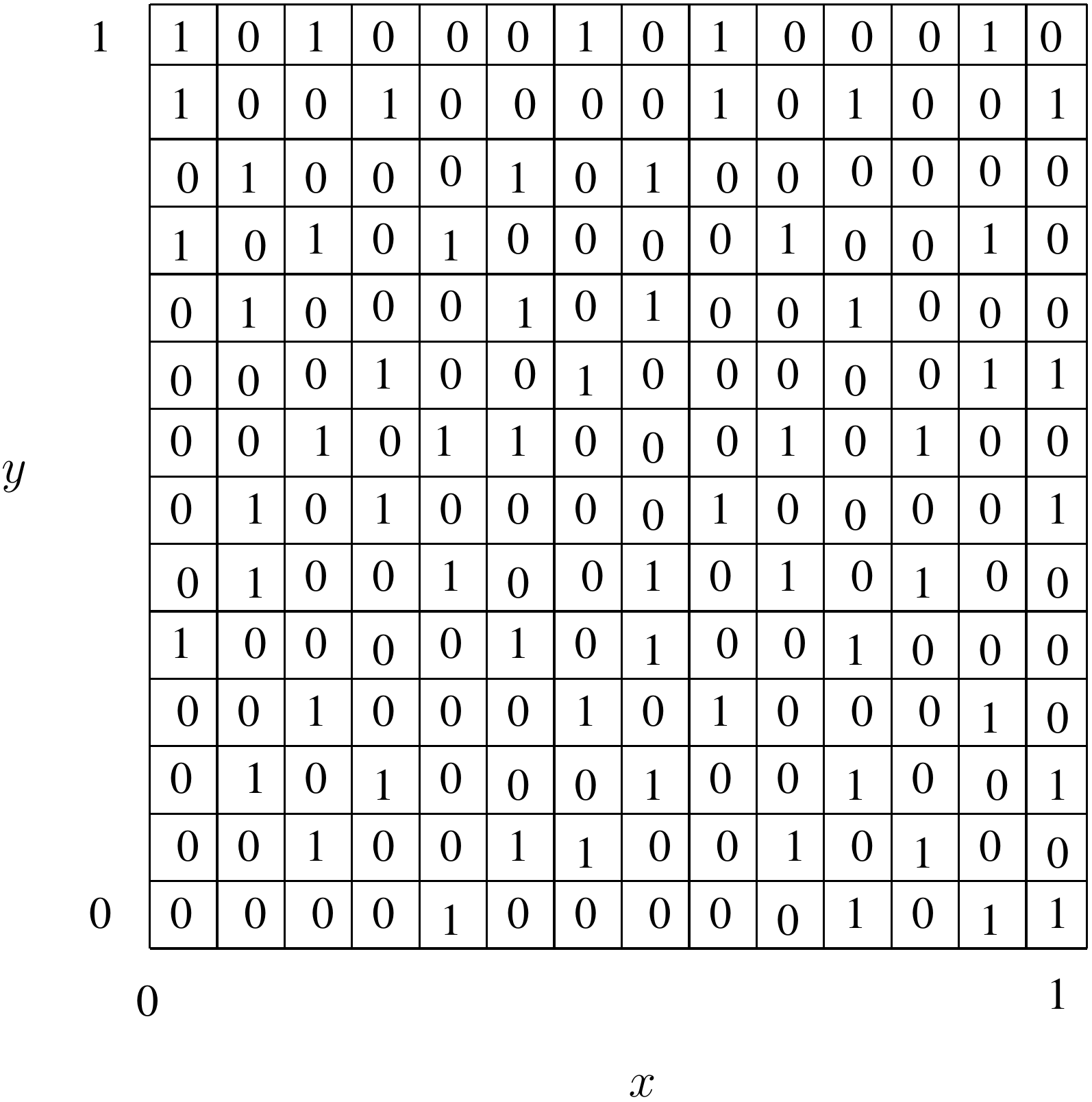}}
\caption{Graph with 14 nodes}
\label{multipartite}
\end{figure}

More generally, a graphon $g\in \G$ is an
arbitrary symmetric
measurable function on $[0,1]^2$ with values in $[0,1]$.
{We define the ``cut metric'' on graphons by}
\be 
{d}(f,g)\equiv \sup_{S,T\subseteq [0,1]}\Big| \int_{S\times
  T}[f(x,y)-g(x,y)]\,dx\,dy\Big|. 
\label{cutmetric}
\ee
Informally, $g(x,y)$ is the probability of an edge between nodes $x$
and $y$, and so two graphons
are called equivalent if they agree up to a 
`node rearrangement', that is, $g(x,y)\sim g(\phi(x),\phi(y))$ where $\phi$ 
is a measure-preserving transformation of $[0,1]$
(see \cite{L} for details). {The cut metric on graphons is invariant under
the action of $\phi$: $d(f\circ \phi, g\circ \phi) = d(f,g)$.} 
We define the cut metric on the quotient space $\tilde \G=\G/\sim$ of `reduced
graphons' to be the infimum of (\ref{cutmetric}) over all representatives
of the given equivalence classes.
$\tilde \G$ is compact in the topology induced by this metric~\cite{L}. 

We now consider the notion of `blowing up' a graph $G$ by replacing each
node with a cluster of $K$ nodes, for some fixed $K=2,3,\ldots$, with
edges inherited as follows: there is an edge between a node in
cluster $V$ (which replaced the node $v$ of $G$) and a node in cluster
$W$ (which replaced node $w$ of $G$) if and only if there is an edge
between $v$ and $w$ in $G$. 
Note that the blowups of a graph are all represented by the same
reduced graphon, and $\Q$ can therefore be considered a graph on
arbitrarily many -- even infinitely many -- nodes, which allows us to
reinterpret Figure~\ref{multipartite} as representing a multipartite graph.
This represents a form of symmetry which we exploit next, and discuss
in Section~\ref{SEC:symmetry}.

The `large scale' features of a graph $G$ on which we focus are the
densities with which various subgraphs $H$ sit in $G$. Assume for
instance that $H$ is a $4$-cycle. We could represent the density of
$H$ in $G$ in terms of the adjacency matrix $A_G$ by
\be
\frac{1}{n(n-1)(n-2)(n-3)} \sum_{w,x,y,z} A_G(w,x)A_G(x,y)A_G(y,z)A_G(z,w),
\ee
where the sum is over distinct nodes $\{w,x,y,z\}$ of $G$. 
For large $n$ this can approximated, within $O(1/n)$, as:

\be
\int_{[0,1]^4} \Q(w,x)\Q(x,y)\Q(y,z)\Q(z,w)\, dw\, dx\, dy\, dz.
\ee
It is therefore useful to define the
density $t_H(g)$ of this $H$ in a graphon $g$ by
\be
t_H(g) = \int_{[0,1]^4} g(w,x)g(x,y)g(y,z)g(z,w)\, dw\, dx\, dy\, dz,
\ee
The density for other subgraphs is defined analogously. 
We note that $t_H(g)$ only depends on the equivalence class of $g$ and is a
continuous function of $g$ with respect to the cut metric 
on reduced graphons. It is an important result of \cite{L} that the densities of subgraphs
are separating: any two reduced graphons with the
same values for all densities $t_H$ are the same.

Our goal is to analyze typical large graphs with
variable edge/triangle constraints in the phase space of Figure
\ref{Razborov-triangle}. Our densities are real numbers, limits of densities
which are attainable in large finite systems, so we begin by softening
the constraints, considering graphs with $n>>1$ nodes and with
edge/triangle densities $(\E^\prime,\T^\prime)$ satisfying
$\E-\delta<\E^\prime<\E+\delta$ and $\T-\delta<\T^\prime<\T+\delta$
for some small $\delta$ (which will eventually disappear.) It is easy to
show that the number of such constrained graphs is of the form $\exp
(sn^2)$, for some $s=s(\E,\T,\delta)>0$ and by a typical graph we mean
one chosen from the uniform distribution on the constrained set.

Getting back to our goal of analyzing constrained uniform
distributions, a related step is to determine the cardinality of the
set of graphs on $n$ vertices subject to the constraints. Constraints
are expressed in terms of a vector $\alpha$ of values of a set $C$ of densities, 
and a small parameter
$\delta$. Denoting the cardinality by $Z_n(\alpha,\delta)$, it was proven
in \cite{RS1, RS2} that $\lim_{\delta\to 0}\lim_{n\to \infty}
(1/n^2)\ln[Z_n(\alpha,\delta)]$ exists; it is called the \emph{constrained
entropy} $s({\alpha})$. As in statistical mechanics this can be
usefully represented via a variational principle.

\begin{theorem}\label{thm:g-variation}
(The variational principle for constrained 
graphs \cite{RS1,RS2})
For any $k$-tuple $\bar H=(H_1,\dots,H_k)$ of subgraphs and $k$-tuple $\alpha=(\alpha_1,\dots,\alpha_k)$ of
numbers,
\be \label{var}
s(\alpha)=\max_{t_{H_j}(q)=\alpha_j} S(g),
\ee
\noindent  where
$S$ is the graphon entropy:
\be \label{Shannon}
S(g)=-\frac12\int_{[0,1]^2} g(x,y)\ln[g(x,y)]+[1-g(x,y)]\ln[1-g(x,y)]\,dx\,dy.
\ee
\end{theorem}

Variational principles such as Theorem \ref{thm:g-variation} are well
known in statistical mechanics \cite{Ru1,Ru2,Ru3}.  
One aspect of such a variational principle in the current
context is not well understood, and that is the uniqueness of the optimizer.
In all known graph examples the
optimizers in the variational principle are unique in $\tilde \G$
except on a lower dimensional set of constraints where
there is a phase transition. Assuming there is a unique optimizer for
(\ref{var}), this optimizer is the limiting constrained uniform distribution.  
We will discuss the question of uniqueness of entropy optimizers 
in Section~\ref{SEC:symmetry}.

Finally we contrast the above formalism with the formalism of exponential random graph
models (ERGM's) mentioned in the Introduction. The latter are widely
used, especially in the social sciences, to model graphs on a fixed,
small number of nodes \cite{N}. They are sometimes considered as
Legendre transforms of the models being discussed in this paper
\cite{RS1}. However, as was pointed out in \cite{CD}, this use of Legendre transform is largely problematic,
since the constrained entropy in these models is neither convex nor concave, and the Legendre
transform is therefore not invertible \cite{RS1}.
As a consequence
the parameters in ERGM's become redundant, and
this confuses any interpretation of phases or phase transitions in such
models. 
\section{Numerical simulation}
\label{SEC:numerics}

We now briefly describe the computational algorithms that we used to obtain the phase portrait sketched in Figure~\ref{phase-diagram}. The details of the algorithms, as well as their benchmark with known analytical results, can be found in~\cite{RRS1}. In the algorithms, we assume that the entropy maximizing graphons are at most $16$-podal, the largest number of podes that our computational power can handle. We then draw random samples from the space of $16$-podal graphons, select those that satisfy, besides the symmetry constraints, the constraints on edge and triangle densities up to given accuracy. We compute the entropy of these selected graphons and take the ones with maximal entropy as the optimizing graphons.

For each given $(\eps,\tau)$ value, we generate a large number of samples such that the optimizing graphons we obtain have entropy values that are within a given accuracy to the true values. The computational complexity of such a procedure is extremely high for high accuracy computations. For each optimizing graphon which we determined with the sampling algorithm, we use it as the initial guess for a Netwon-type local optimization algorithm, more precisely a sequential quadratic programming (SQP) algorithm, to search for local improvements. Detailed analysis of the computational complexity of the sampling algorithm and the implementation of the SQP algorithm are documented in ~\cite{RRS1}. The results of the SQP algorithms are the candidate optimizing graphons that we show in the rest of this section.

\begin{figure}[!htb]
\centering
\includegraphics[width=0.675\textwidth,height=0.615\textwidth]{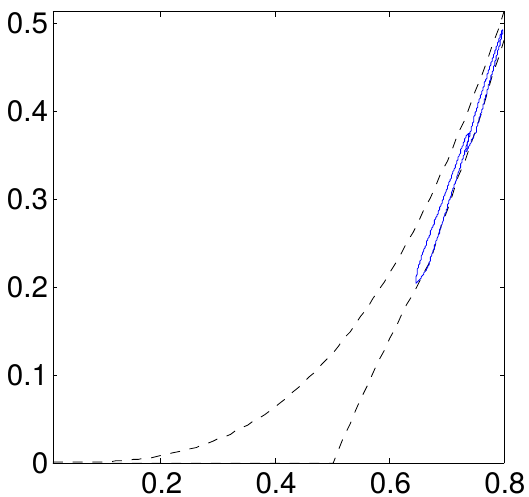}
\caption{Boundaries of the $A(3,0)$ and the $A(4,0)$ phases, in solid blue lines, as determined by numerical simulations. 
The black dashed lines are the Erd\H{o}s-R\'{e}nyi curve (upper) and Razborov scallops (bottom).}
\label{FIG:N-0 Boundary}
\end{figure}

\paragraph{The $A(n,0)$ phase boundaries.} 
In the first group of numerical simulations we try to determine the
boundaries of two of the completely symmetric phases, the $A(3,0)$
phase and the $A(4,0)$ phase. As we pointed out in the Introduction,
for a given $(\eps, \tau)$ value in $A(n,0)$ we know analytically the unique
expression for the optimal $A(n,0)$ graphon (see, for instance,
equation (\ref{a-and-b})) and the corresponding entropy. Therefore, a given
point $(\eps, \tau)$ is outside of the $A(n, 0)$ phase if we can find
a graphon that gives a larger entropy value, and is within
the $A(n, 0)$ phase if we can not find a graphon that has a
larger value. We use this strategy to determine the boundaries
of the $A(3,0)$ and the $A(4, 0)$ phases. In 
Figure~\ref{FIG:N-0 Boundary} we show the boundaries we determined
numerically. Note that since these boundaries are determined using a
mesh on the $(\eps, \tau)$ plane, they are only accurate
up to the mesh size in each direction. 
For better visualization, we reproduced
Figure~\ref{FIG:N-0 Boundary} in a rescaled coordinate system, $(\eps,
\tau'=\tau-\eps(2\eps-1))$, in Figure~\ref{FIG:N-0 Boundary
  Rescaled}. 
\begin{figure}[!htb]
\centering
\includegraphics[width=0.45\textwidth]{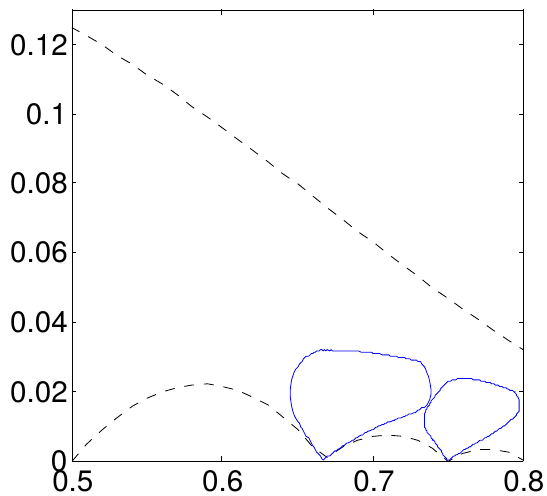}
\caption{Same as in Figure~\ref{FIG:N-0 Boundary} except that the coordinates $(\eps, \tau'=\tau-\eps(2\eps-1))$ are used, and $\eps$ is restricted to
$[0.5, 0.8]$.}
\label{FIG:N-0 Boundary Rescaled}
\end{figure}

\begin{figure}[!htb]
\centering
\includegraphics[width=0.22\textwidth]{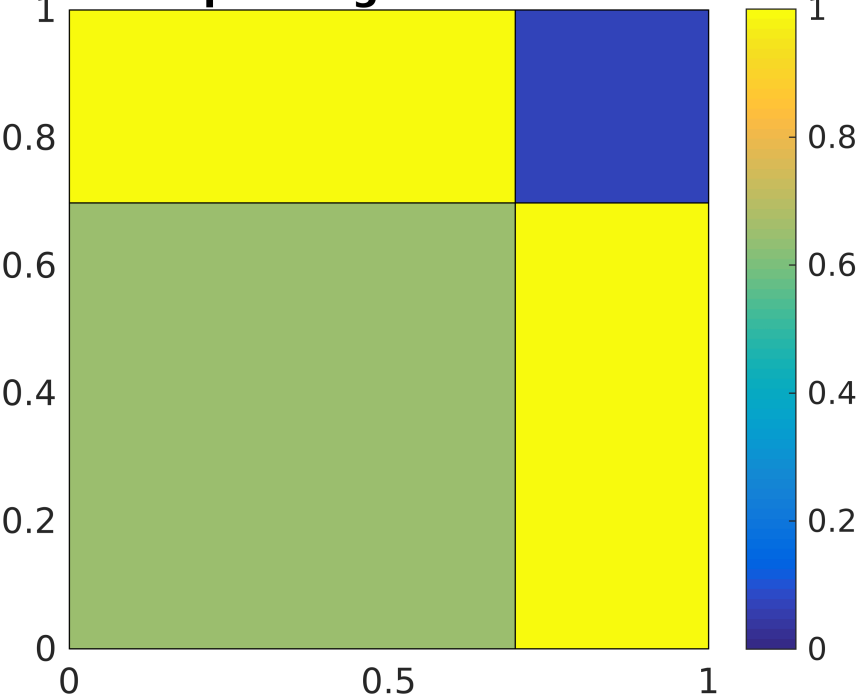}
\includegraphics[width=0.22\textwidth]{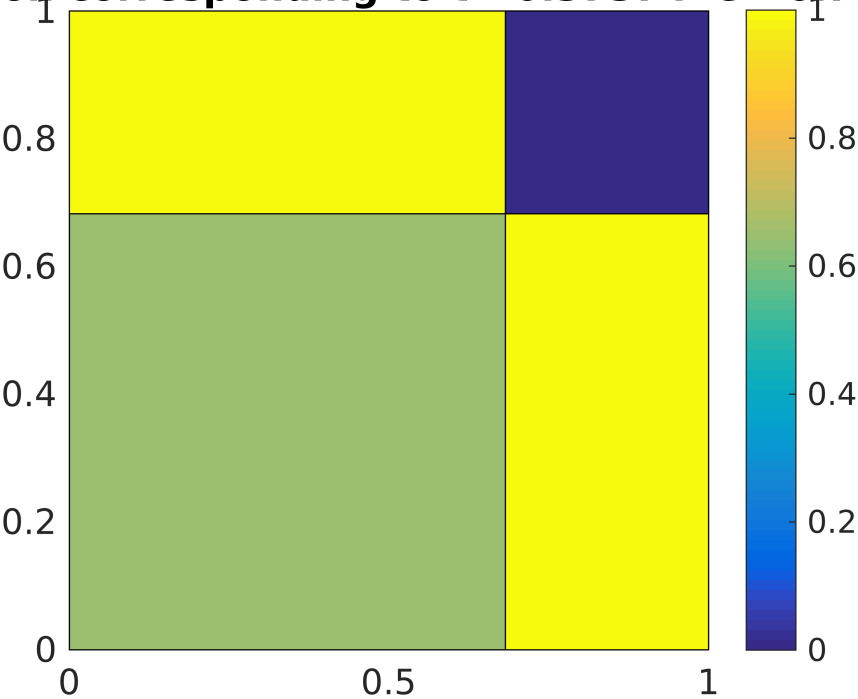}
\includegraphics[width=0.22\textwidth]{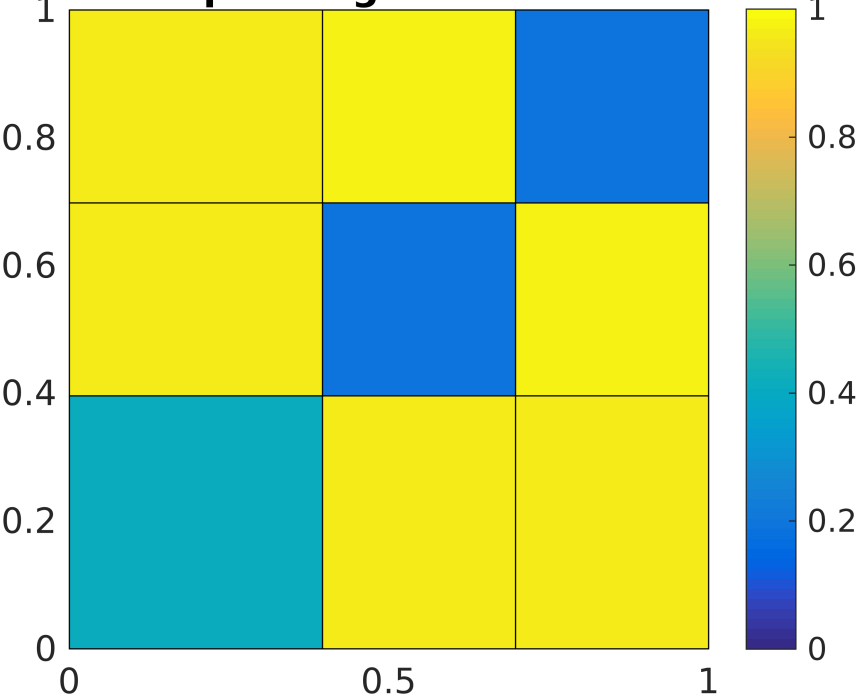}
\includegraphics[width=0.22\textwidth]{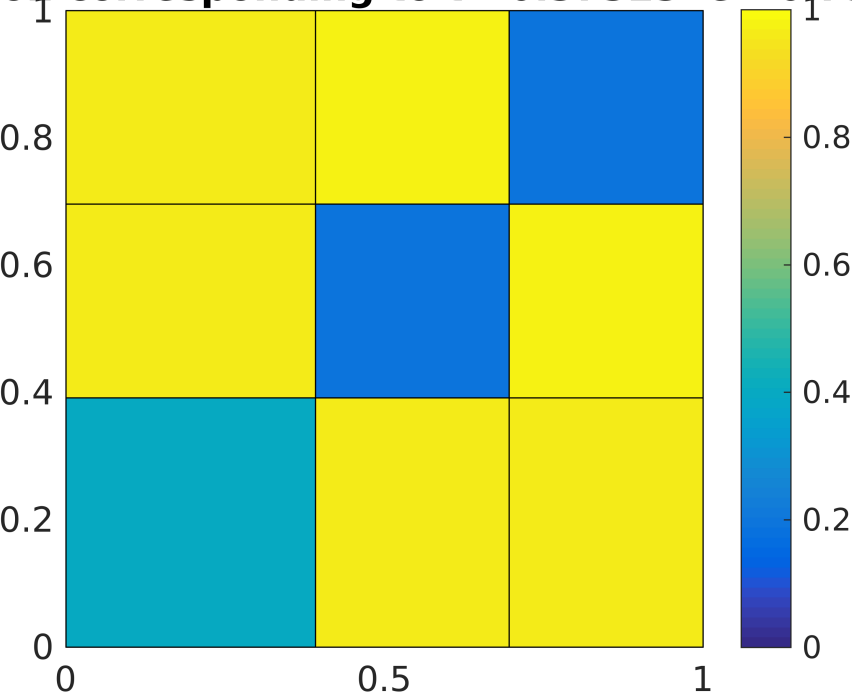}\\
\includegraphics[width=0.22\textwidth]{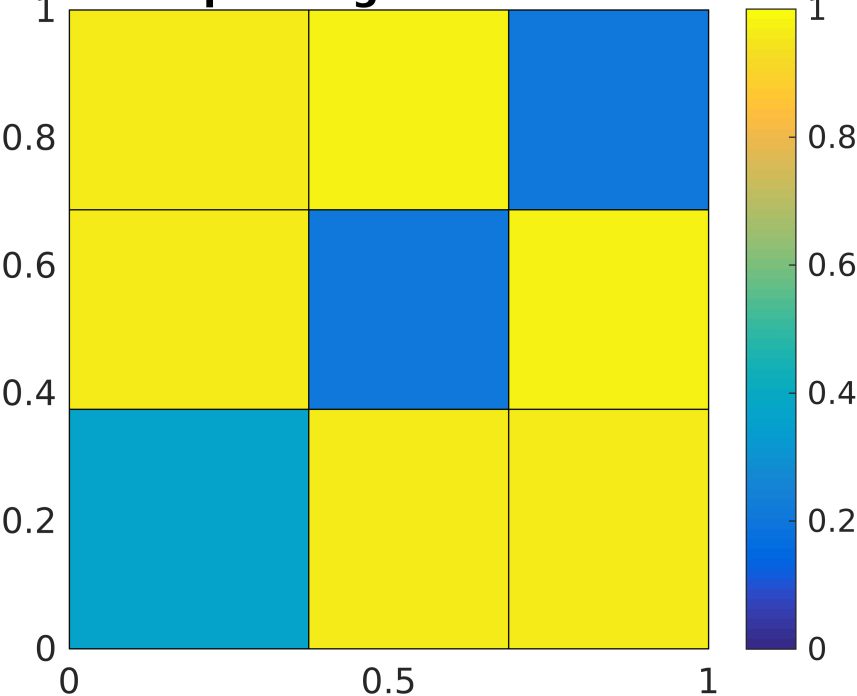}
\includegraphics[width=0.22\textwidth]{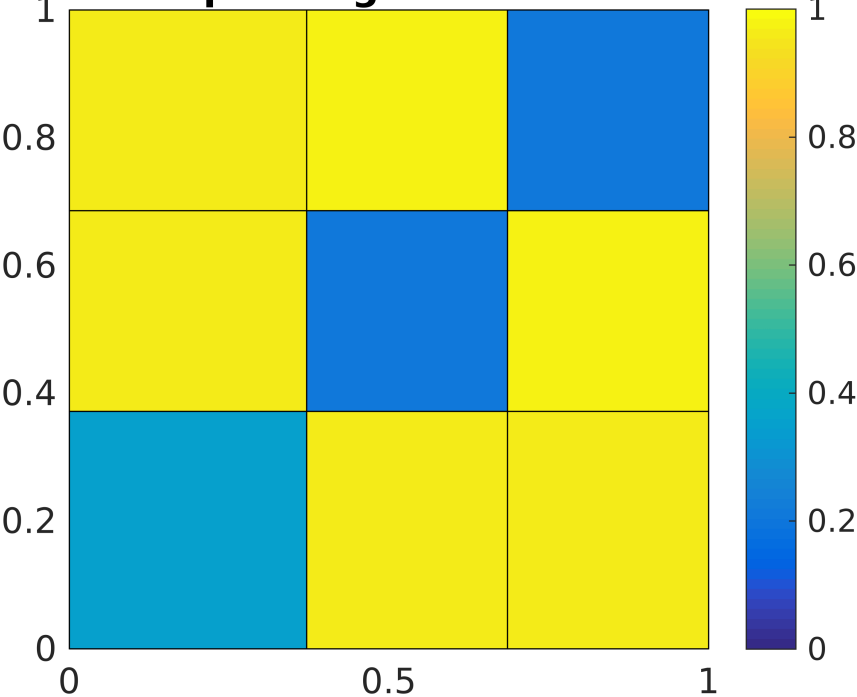}
\includegraphics[width=0.22\textwidth]{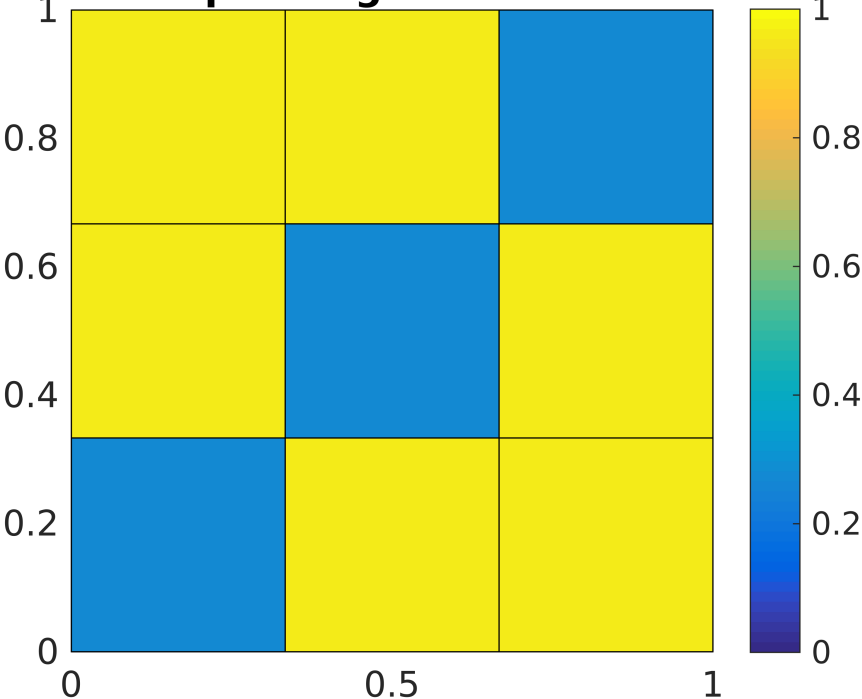}
\includegraphics[width=0.22\textwidth]{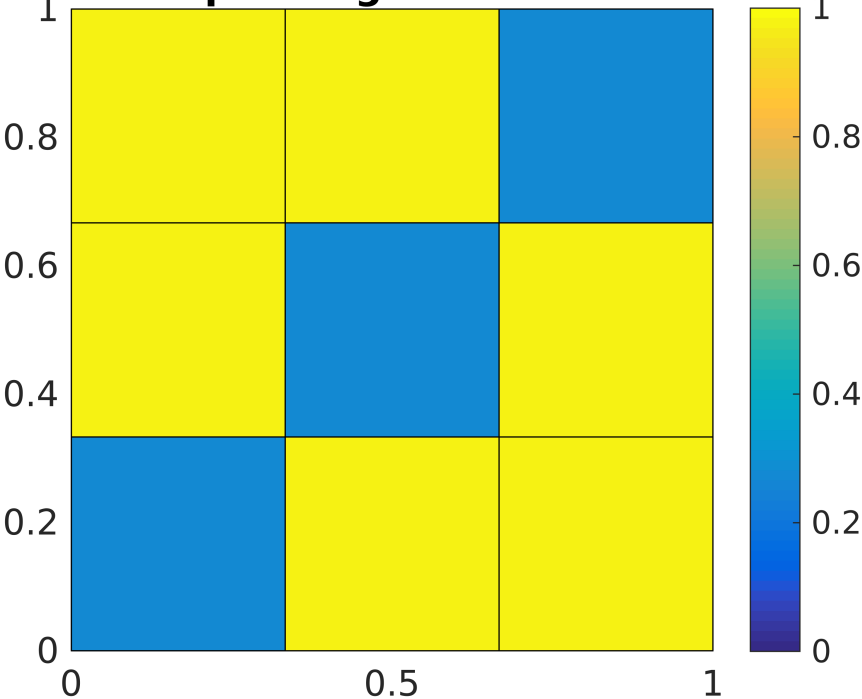}\\
\includegraphics[width=0.22\textwidth]{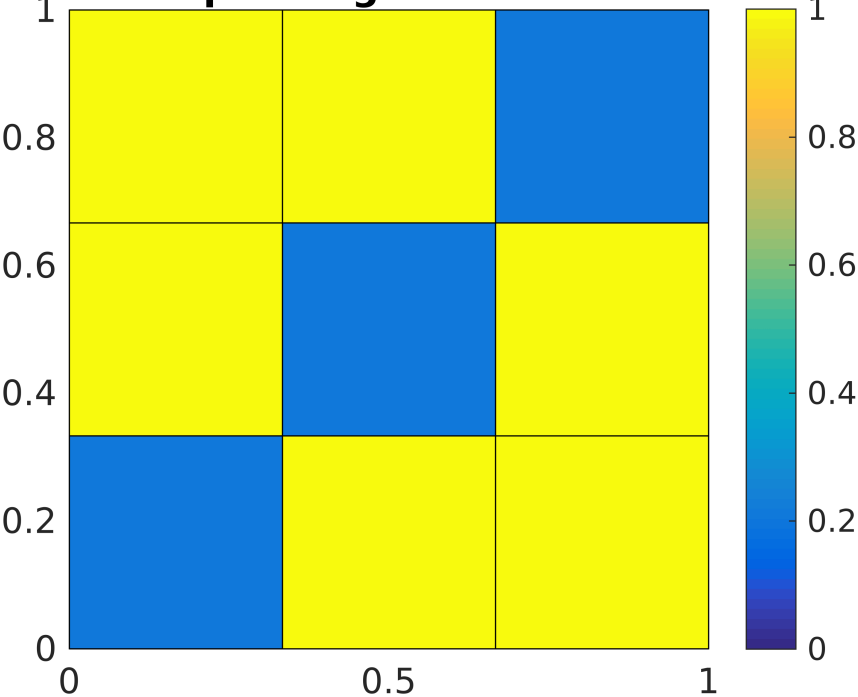}
\includegraphics[width=0.22\textwidth]{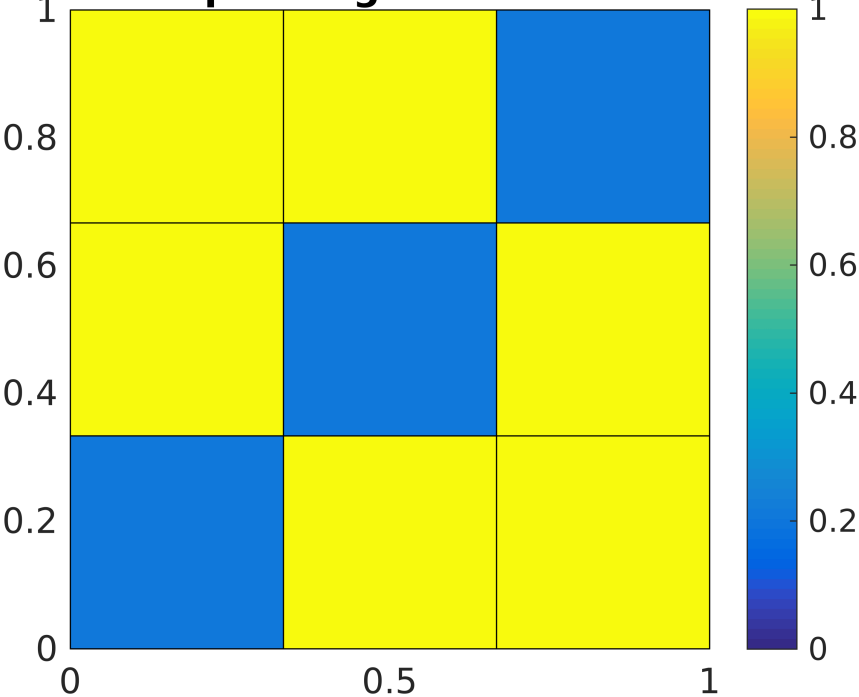}
\includegraphics[width=0.22\textwidth]{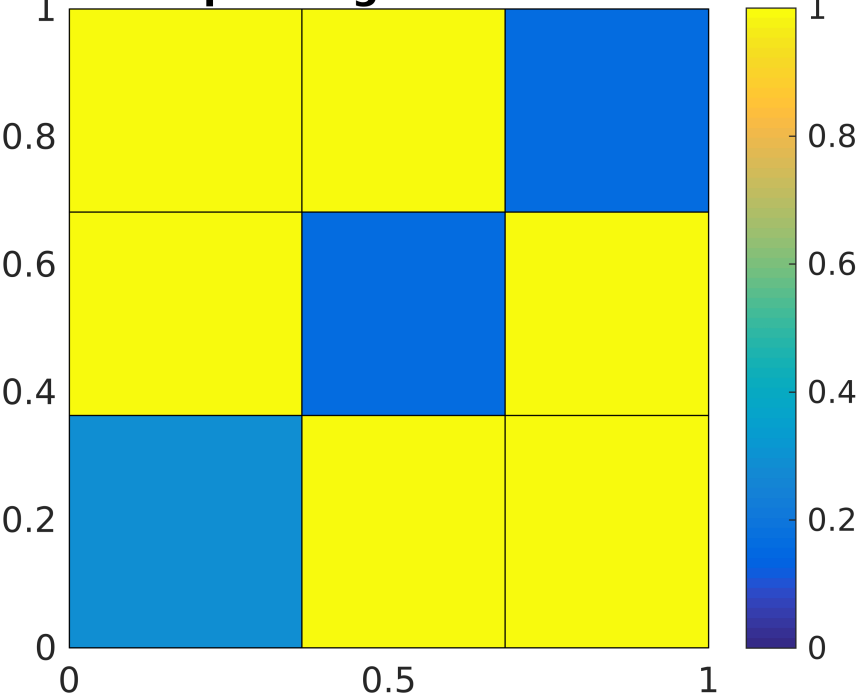}
\includegraphics[width=0.22\textwidth]{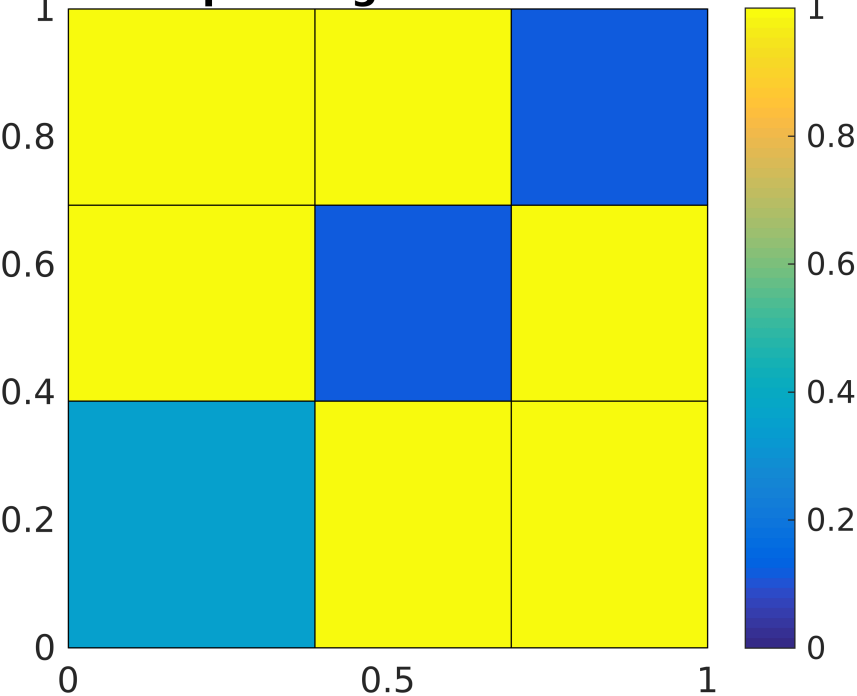}\\
\includegraphics[width=0.22\textwidth]{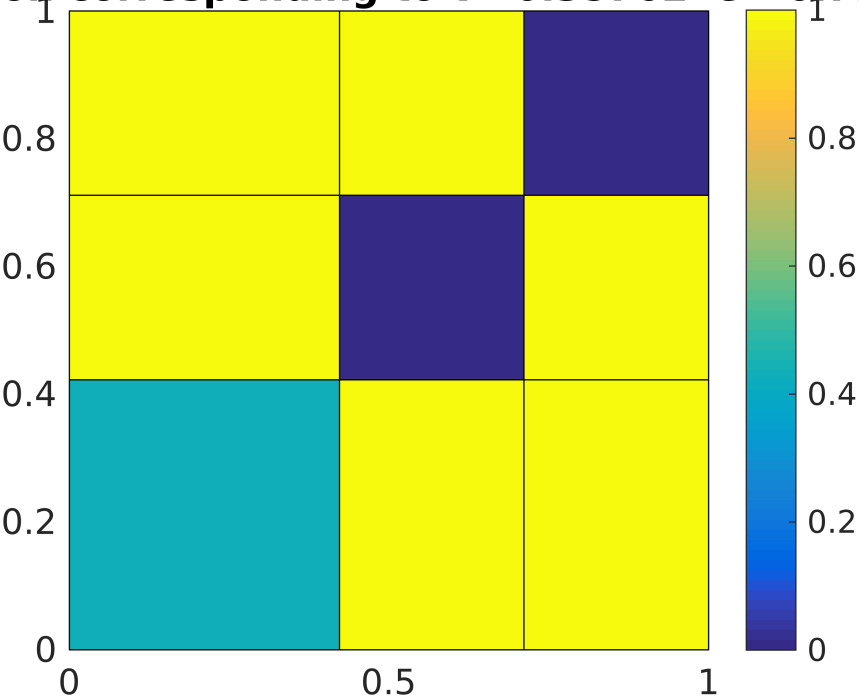}
\includegraphics[width=0.22\textwidth]{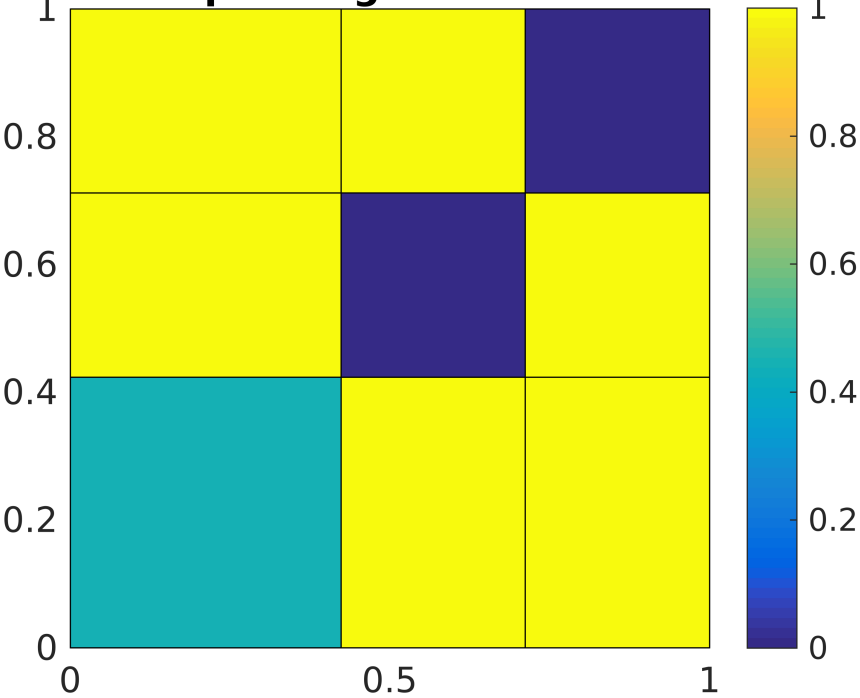}
\includegraphics[width=0.22\textwidth]{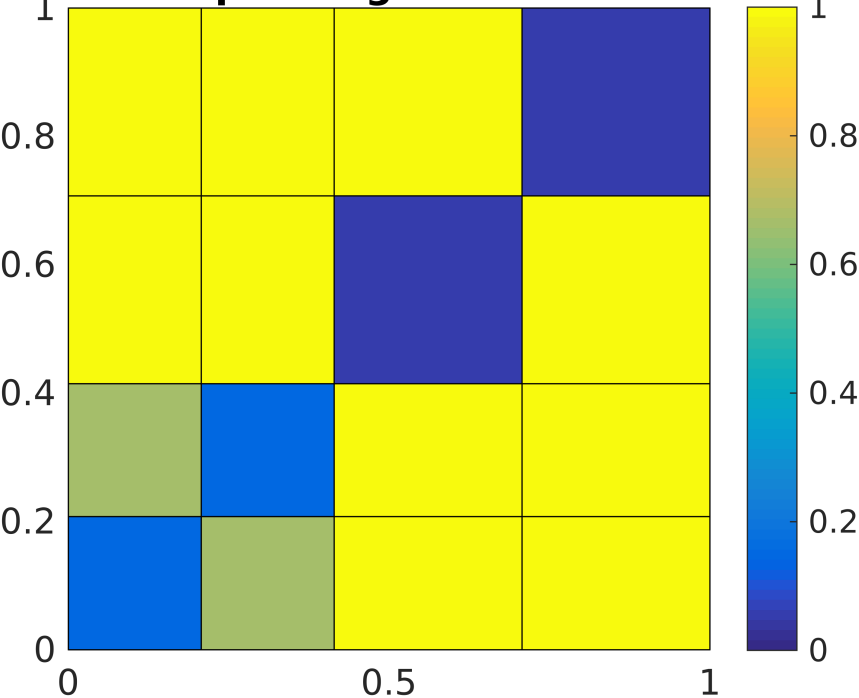}
\includegraphics[width=0.22\textwidth]{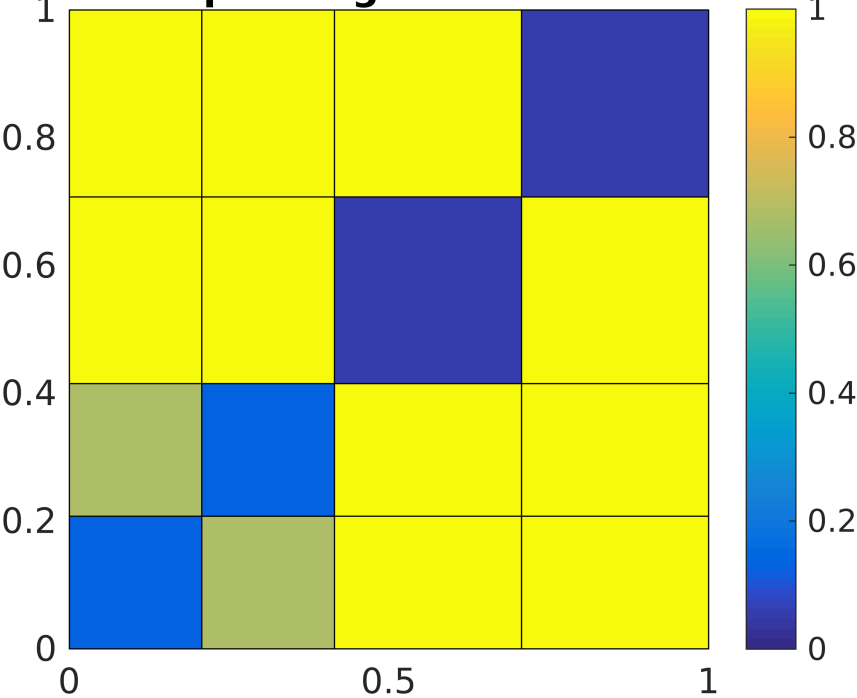}\\
\includegraphics[width=0.22\textwidth]{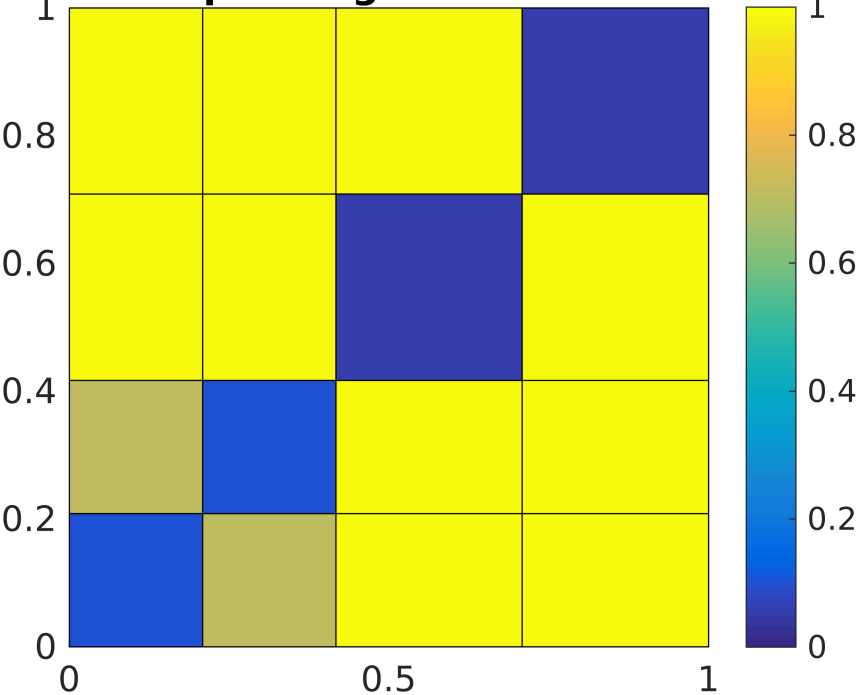}
\includegraphics[width=0.22\textwidth]{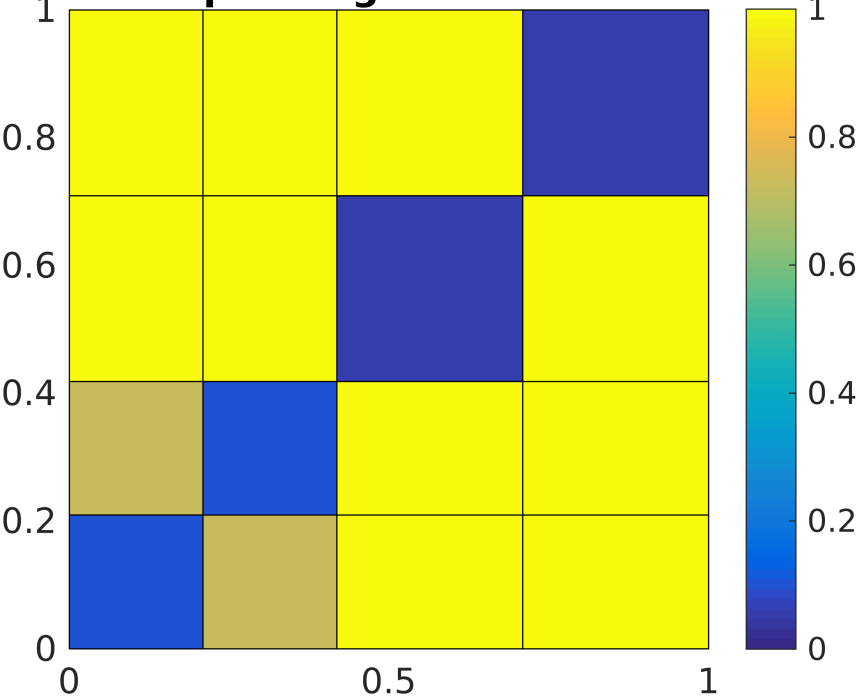}
\includegraphics[width=0.22\textwidth]{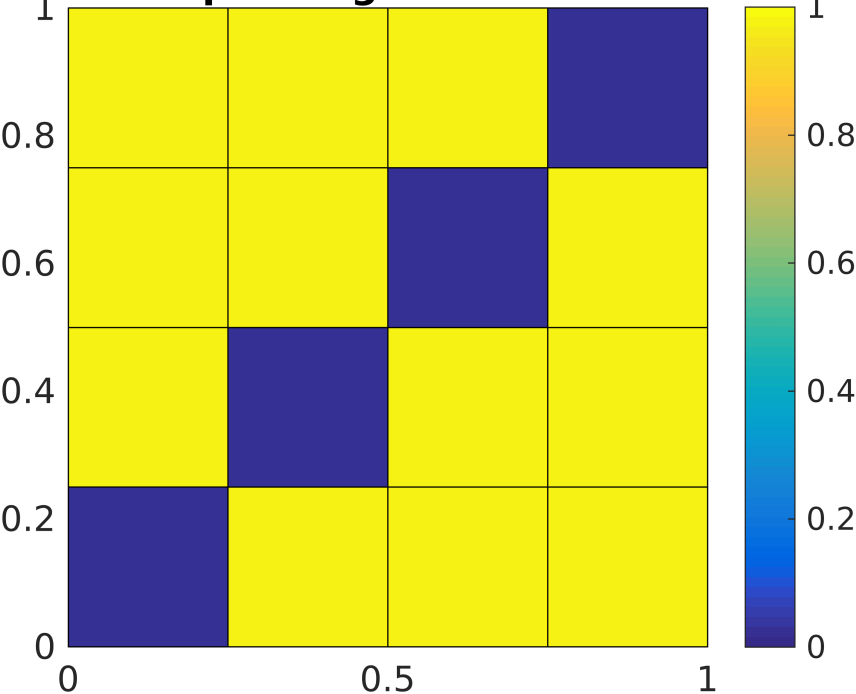}
\includegraphics[width=0.22\textwidth]{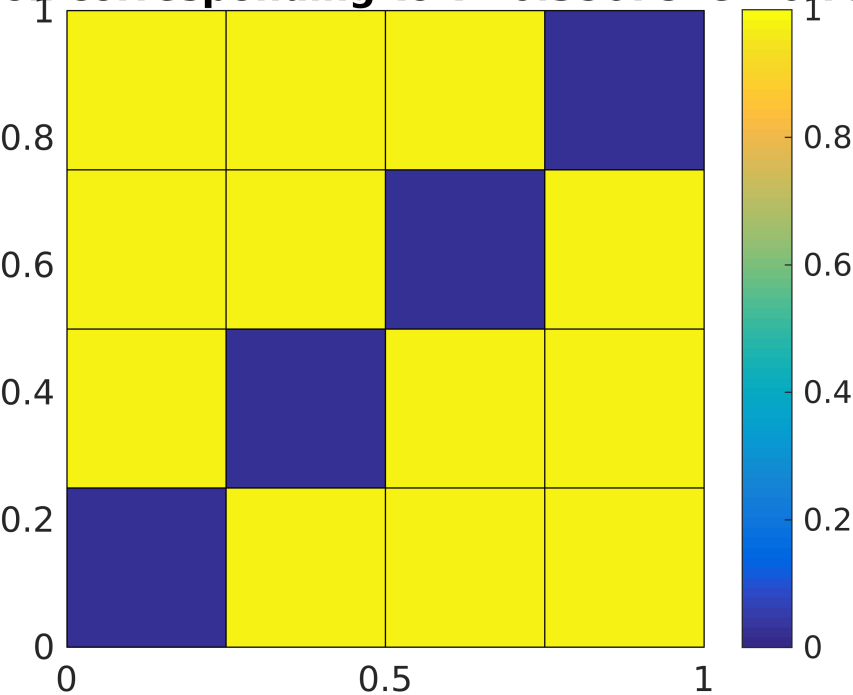}\\
\includegraphics[width=0.22\textwidth]{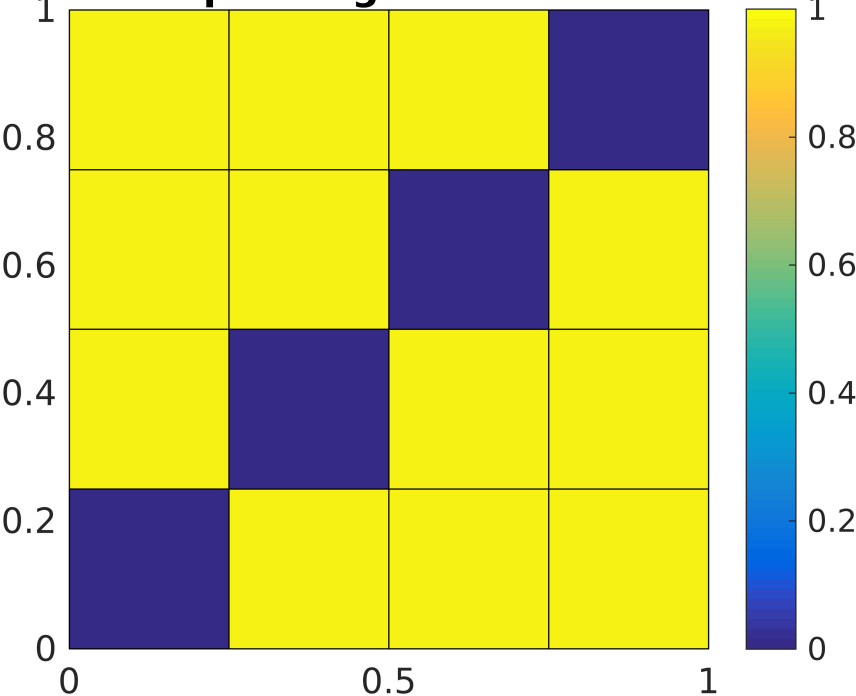}
\includegraphics[width=0.22\textwidth]{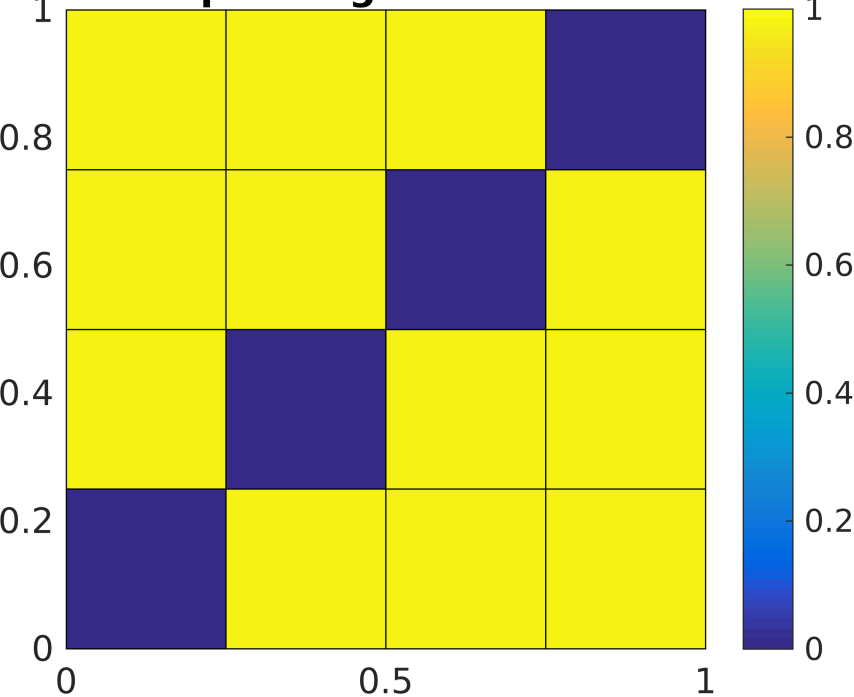}
\includegraphics[width=0.22\textwidth]{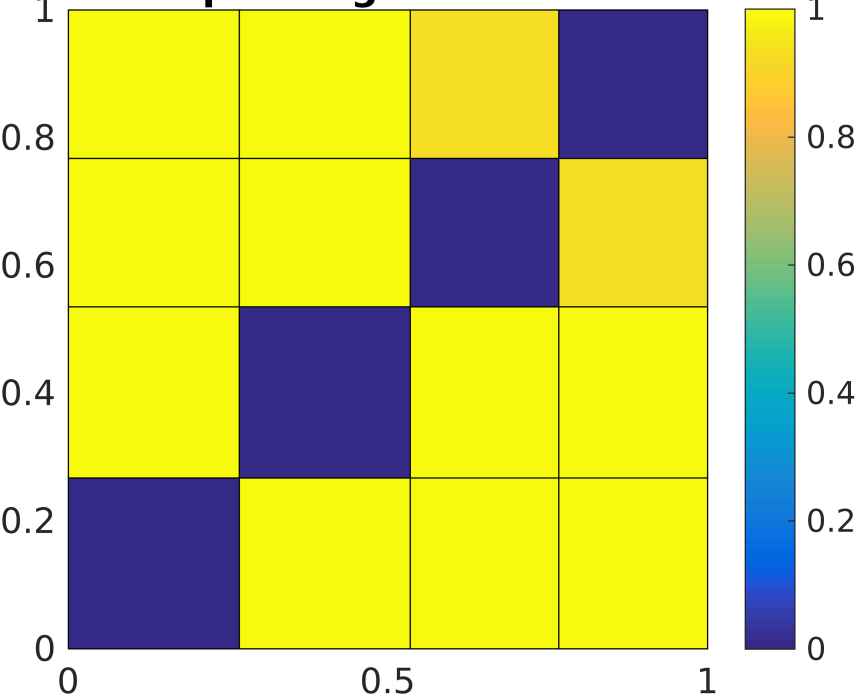}
\includegraphics[width=0.22\textwidth]{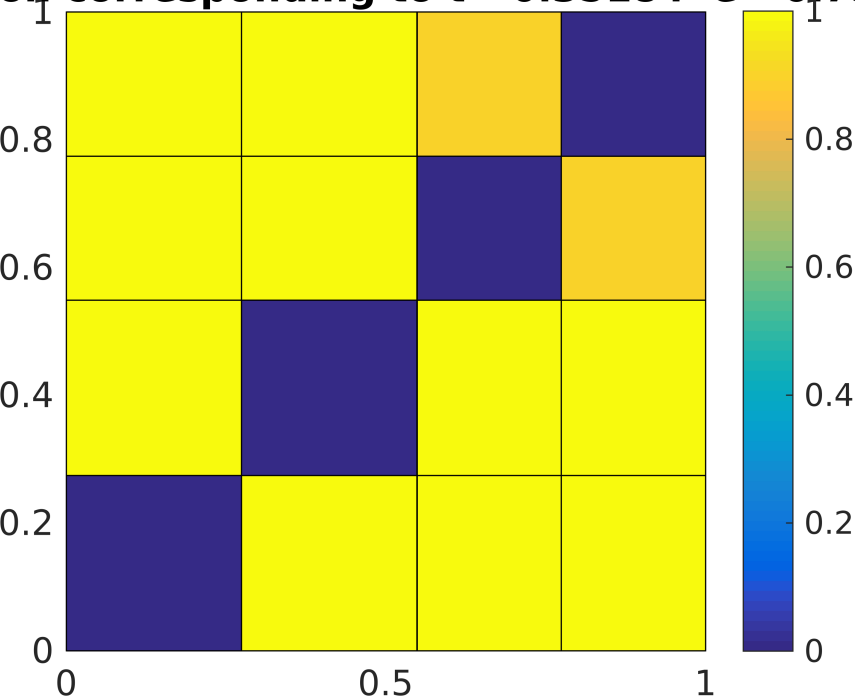}
\caption{The optimizing graphons at the phase boundaries, cut as $\T$
  decreases at fixed $\E=0.735$. The first two columns are before the
  transition and the last two columns are after the transition. Rows,
  top to bottom, show respectively the following transitions: (i)
  $B(1,1)\to B(2,1)$, (ii) $B(2,1)\to A(3,0)$, (iii) $A(3,0)\to
  B(2,1)$, (iv) $B(2,1)\to C(2,2)$, (v) $C(2,2)\to A(4,0)$, and (vi)
  $A(4,0)\to C(2,2)$. We discuss in the text the $\T$ values surrounding the
  transitions, corresponding to the middle two columns.}
\label{FIG:Line E0p735}
\end{figure}
\paragraph{Phases along line segment $(0.735, \tau)$.} 
In the second group of numerical simulations, we provide some
numerical evidence to support our conjecture on phases depicted
in Figure~\ref{phase-diagram}. Consider the optimizing graphons
in the phases which the vertical line segment $(0.735, \tau)$:
$\tau\in(0.3504, 0.3971)$, cuts through. (Figure~\ref{phase-diagram} is
quite crude: for an accurate representation of the
boundaries of $A(3,0)$ and $A(4,0)$ see Figure~\ref{FIG:N-0
  Boundary} .) From top to bottom, the line
cuts through: $B(1,1)$, $B(2,1)$, $A(3,0)$, $B(2,1)$, $C(2,2)$,
$A(4,0)$, and $C(2,2)$ phases. In Figure~\ref{FIG:Line E0p735}, we show the
optimizing graphons in phases before and after each transition along
the line. Let us mention here two obvious discontinuous
transitions, the first being the transition from $B(1,1)$ to
$B(2,1)$ at about $\tau=0.3737$, and the second being the transition
from $B(2,1)$ to $C(2,2)$ at about $\tau=0.3574$. 
{The derivative $\partial s/\partial \T$ of the entropy with respect to 
$\T$ exhibits jumps at these transitions}, as seen in
Figure~\ref{FIG:Line E0p735 Entropy}. A theoretical analysis of the transitions is
presented in the next section.

\paragraph{Phases along line segment $(0.845, \tau)$.} 
We performed similar simulations to reveal phases along the line
segment $(0.845, \tau)$, $\tau\in(0.5842, 0.6034)$. The phases, from
top to bottom, should be, respectively $B(1,1)$, $B(2,1)$, $B(3,1)$,
$B(4,1)$, $A(6,0)$, $B(4,1)$, $B(5,1)$, and
$C(5,2)$. Our simultation did not capture the last two phases, which lie
extremely close to the lower boundary of the phase space; 
representive optimizing graphons for the others
are shown in Figure~\ref{FIG:Line E0p845}. Also, due to limitations in 
computational power we were not able to resolve the transitions
between the phases as accurately as in the previous case, that is, those
in Figure~\ref{FIG:Line E0p735}.

\paragraph{Remark on our conjecture on the phase space structure.} We now briefly 
summarize the evidence behind our conjecture of the phase portrait in
Figure~\ref{phase-diagram}. First of all, by simulations and proofs in
several models, in particular this one, we had found that in the
interior of phase spaces optimizing graphons have always been found to
be multipodal. In this model furthermore, the only place we find more
than 3 podes is for $\eps>0.7$ and near the scalloped boundary.
(Recall that in all the numerical simulations in this paper we assume
that the graphons have $16$ or fewer podes, even though where we simulate
they always end up having many fewer than $16$ podes). Second, for each point
$(\eps, \tau)$ in the phase space except for $\eps > 0.75$ and close to
the scallops), our algorithm, even though computationally expensive, could
determine the optimal graphons up to relatively high accuracy. 
The computational cost is in general tolerable; in the exceptional
region it is too expensive to find enough graphons which satisfy all
the constraints to get the accuracy we wanted.  (In more detail, we use the techniques explained in~\cite{RRS1} with edge and triangle constraint intervals of size $10^{-9}$, which determines entropy to order $10^{-6}$, from which we determine our optimal graphons.)
Third, within each phase away from the scallops our computational power allows us to
perform simulations on relatively fine meshes of $(\eps, \tau)$. This
is how we determined the boundary of the $A(3,0)$ and $A(4,0)$ phases
as well as the transitions shown in Figure~\ref{FIG:Line E0p735
  Entropy}, for instance. In more detail, consider the middle two graphons in each row in
Figure~\ref{FIG:Line E0p735}, which straddle the transitions. They all
have edge density $\E=0.735$. In the first row they have triangle
densities $0.37373880$ and $0.37368265$; in the second row the triangle densities
are $0.37171725$ and $0.37154879$; in the third row they are $0.36138489$
and $0.36132874$; in the fourth row they are $0.35745410$ and
$0.35739795$; in the fifth they are $0.35706102$ and $0.35694871$; and
in the last row they are $0.35475870$ and $0.35318639$.

\begin{figure}[!htb]
\centering
\includegraphics[width=0.5\textwidth]{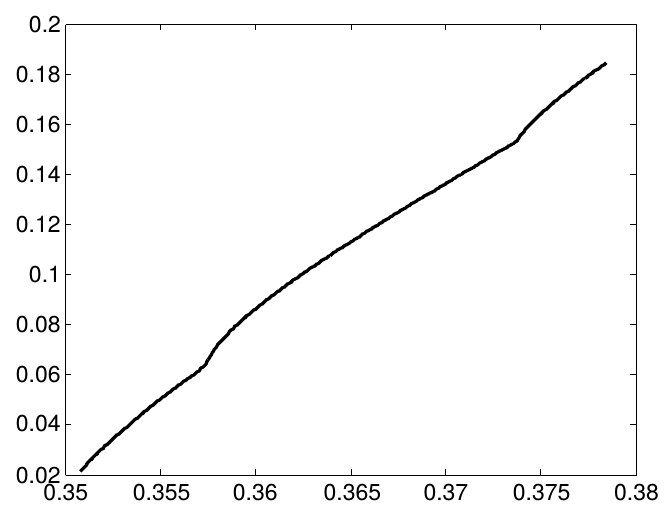}
\includegraphics[width=0.2\textwidth,height=0.38\textwidth]{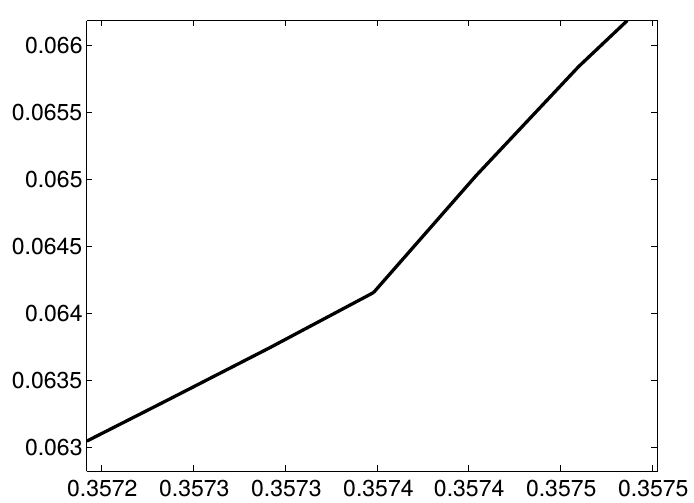}
\includegraphics[width=0.2\textwidth,height=0.38\textwidth]{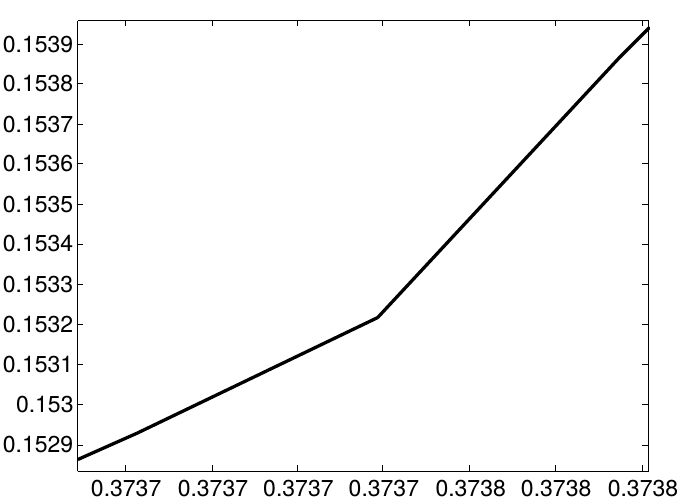}
\caption{Left: entropy $s$ as a function of $\tau$ along the line segment $(0.735, \tau)$, $\tau\in(0.3504, 0.3971)$; Middle: zoom  of $s(\tau)$ around the $B(2,1)\to C(2,2)$ phase transition at $\tau=0.3574$; Right: zoom  of $s(\tau)$  around the $B(1,1)\to B(2,1)$ phase transition at  $\tau=0.3737$.}
\label{FIG:Line E0p735 Entropy}
\end{figure}

\begin{figure}[!htb]
\centering
\includegraphics[width=0.3\textwidth]{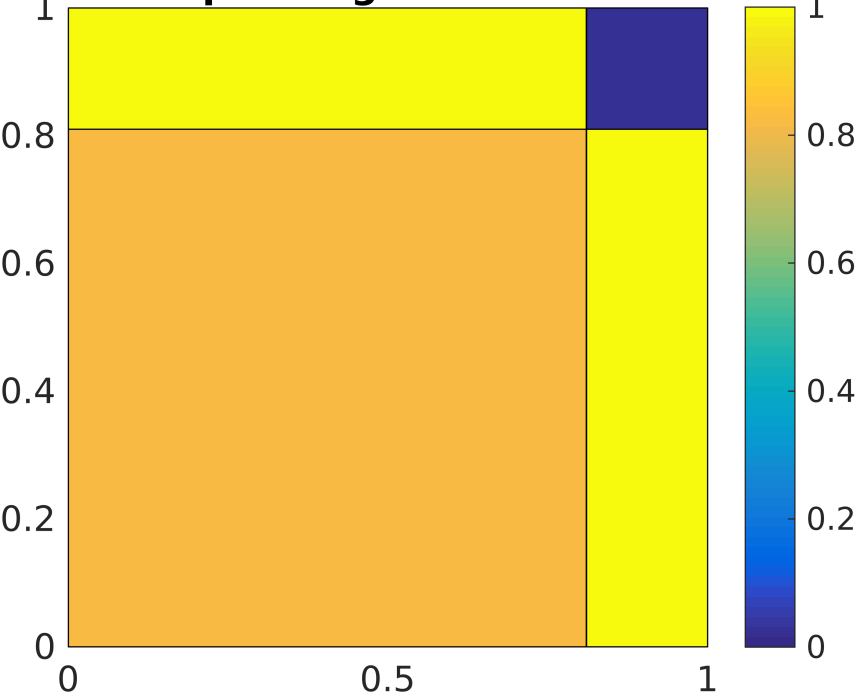}
\includegraphics[width=0.3\textwidth]{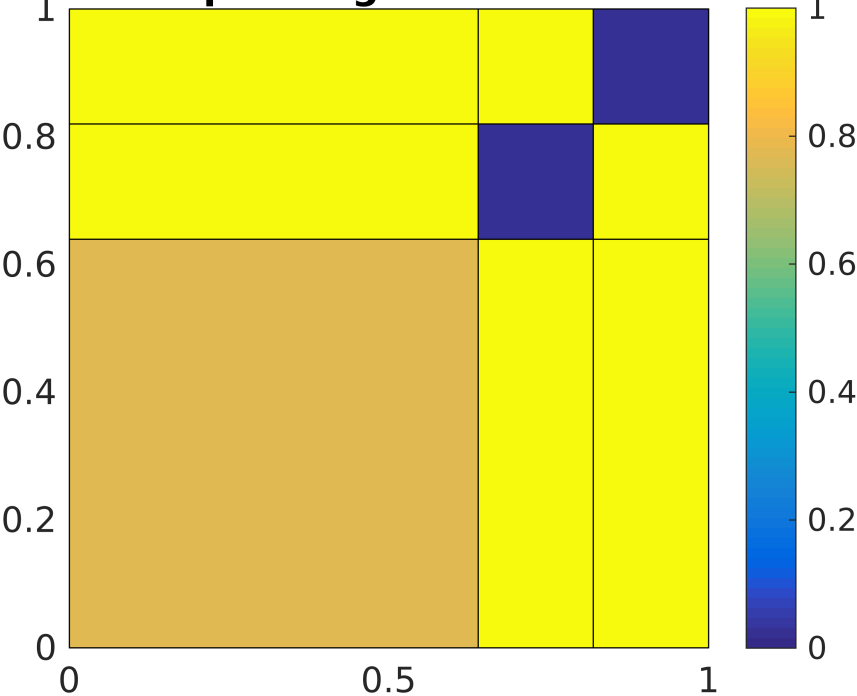}
\includegraphics[width=0.3\textwidth]{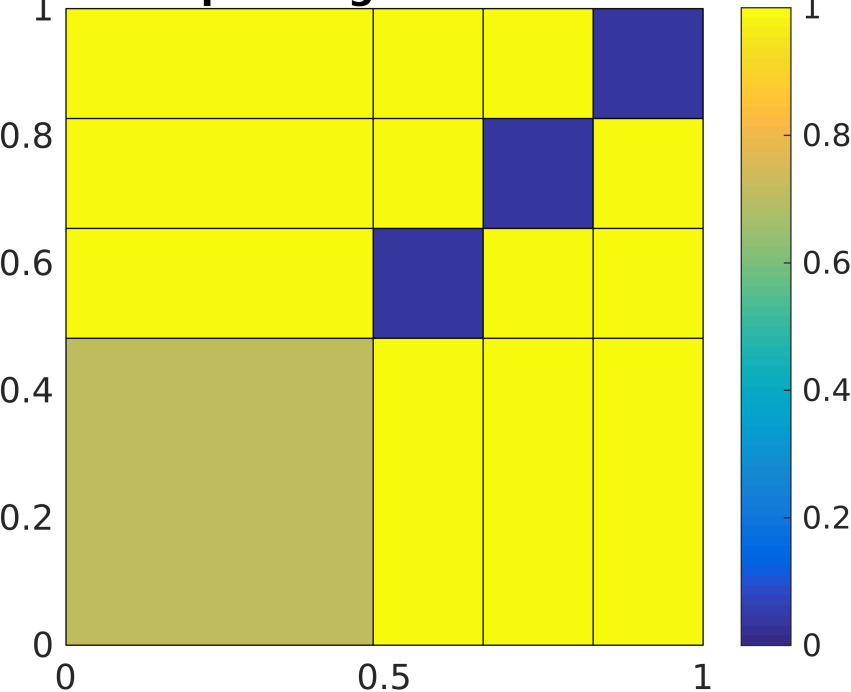}\\
\includegraphics[width=0.3\textwidth]{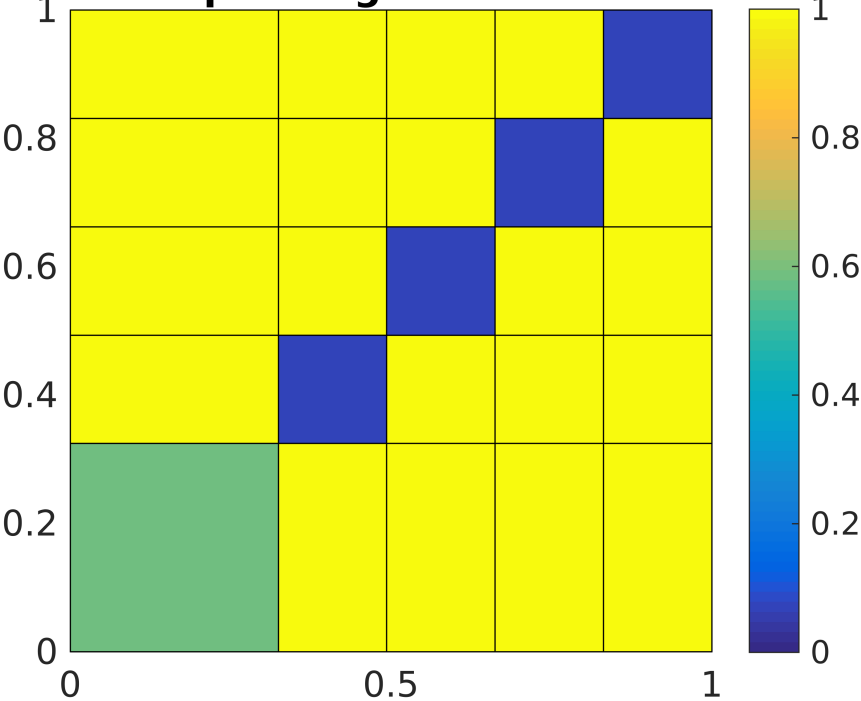}
\includegraphics[width=0.3\textwidth]{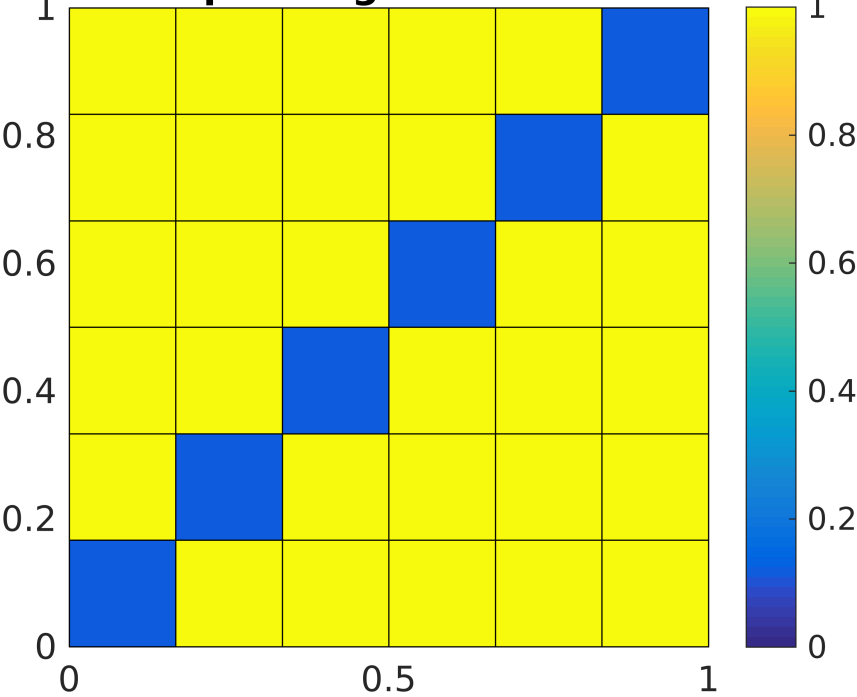}
\includegraphics[width=0.3\textwidth]{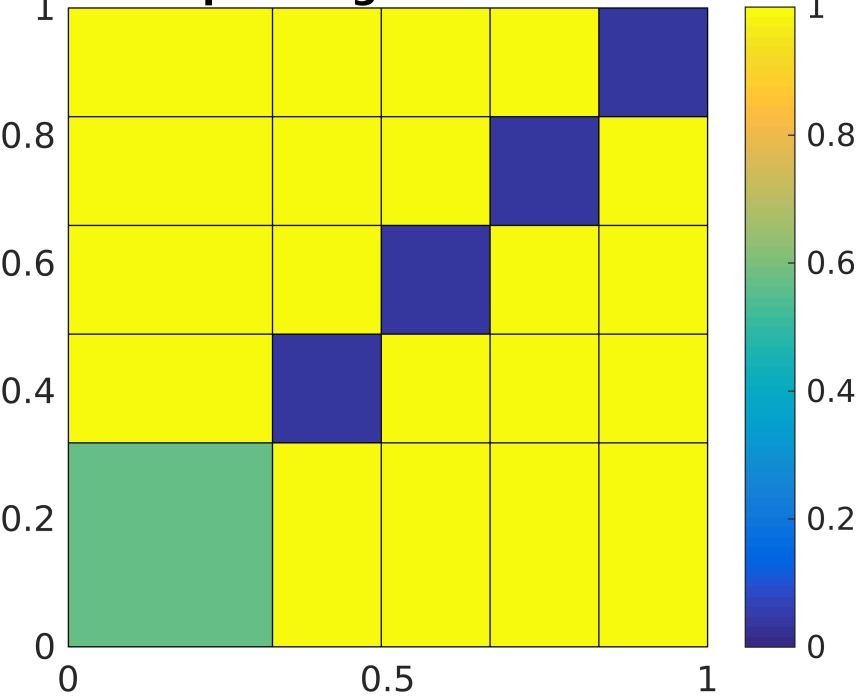}
\caption{Typical optimizing graphons in the phases that the line
  segment $(0.845, \tau)$, $\tau\in(0.5842, 0.6034)$ cuts
  through. From top left to bottom right are respectively graphons in
  phases $B(1,1)$, $B(2,1)$, $B(3,1)$, $B(4,1)$, $A(6,0)$
  and $B(4,1)$, respectively.}
\label{FIG:Line E0p845}
\end{figure}

\section{Analysis of transitions}
\label{SEC:transitions}

As noted above, the numerical simulations indicate that all phase
transitions below the Erd\H{o}s-R\'{e}nyi curve with the exception of
$A(n,0) \leftrightarrow B(n-1,1)$, occur discontinuously.  At each
other transition, the optimizing graphon jumps from one form to
another, and so the densities of certain subgraphs also jump. In this
section, we prove that certain of these transitions \emph{can only
  occur discontinuously}.

\begin{theorem}\label{thm:no-continuous-theorem-1}
 Except at a finite number of values of the edge density $\E$, 
there cannot be a 
continuous transition from a $B(1,1)$ bipodal phase to a $B(2,1)$ or $C(1,2)$ 
tripodal phase. 
\end{theorem}

Theorem~\ref{thm:no-continuous-theorem-1} can be generalized to consider all $B \leftrightarrow C$ transitions:

\begin{theorem}\label{thm:no-continuous-theorem-2}
Except at a finite number of values of the edge density $\E$, there cannot be a continous transition from
a $B(m,1)$ to a $C(m,2)$ phase. 
\end{theorem}

Assuming that our conjectured phase portrait is correct, this leaves the
$B(m,1) \leftrightarrow B(m-1,1)$, $C(n-2,2) \leftrightarrow A(n,0)$, 
$A(n,0) \leftrightarrow B(n-2,1)$ and
$A(n,0) \leftrightarrow B(n-1,1)$ transitions for us to consider. 

\begin{theorem}\label{thm:3}
For $m>2$, there cannot be a continuous transition from a $B(m,1)$
phase to a $B(m-1,1)$ phase, and there cannot be a continuous transition
from an $A(n,0)$ phase to a $B(n-2,1)$ phase.  
\end{theorem}

To summarize (see Figure \ref{phase-diagram}) we have proven 
that transitions cannot be continuous
between: any two $B$ phases; any
$A(n,0)$ and $B(n-2,1)$ phases; any $B$ and $C$ phases except perhaps
at finitely many values of $\E$.
As noted earlier, transitions between $A$ and $C$ phases appear to
always be discontinuous, and transitions between any $A(n,0)$ and
$B(n-1,1)$ phases
appear to always be continuous. However, we currently lack a
proof for these last two claims, although in \cite{RRS2} there is a possible 
path to prove continuity for $A(2,0) \leftrightarrow B(1,1)$. 

For completeness we note that the two transitions across the
Erd\H{o}s-R\'{e}nyi curve, $A(2,0) \leftrightarrow
F(1,1)$ and $B(1,1)\leftrightarrow
F(1,1)$, are continuous: for $(\E,\T)$ on the Erd\H{o}s-R\'{e}nyi
curve the graphon with constant value $\E$ is easily seen to be the
unique entropy-optimizer, and the optimizers from each side must approach
it in $L^2$ and therefore in cut metric.

 \noindent{\bf Proof of Theorem \ref{thm:no-continuous-theorem-1}.}
The $C(1,2)$ phase is in the $C$ family and appears to the left of
the $A(3,0)$ phase, while $B(2,1)$ is in the $B$
family and appears to the right of $A(3,0)$. 
For purposes of this proof, however, they are indistinguishable. 
All that matters is that there are three ``podes'', with a symmetry 
swapping two of them.

The proof has three steps:
\begin{enumerate}
\item Showing that the Lagrange multipliers (see definition below) would have to diverge at 
a continuous transition.
\item Showing that for a bipodal optimizing graphon, divergent Lagrange 
multipliers can only occur at the ``natural boundary'' of the $B(1,1)$ phase,
namely at the minimum possible value of $\T$ achievable by a bipodal 
graphon for the given value of $\E$. 
\item Showing that on the natural boundary, there exists a tripodal graphon
with the given values of $(\E,\T)$ and with higher entropy than any 
bipodal graphon. That is, showing that the natural boundary of the $B(1,1)$ 
phase actually lies within a tripodal phase. This step requires that a certain
analytic function of $\E$ be nonzero. Since analytic functions (that
aren't identically zero) can only have finitely many roots in a compact 
interval, this step of the proof can break down at finitely many values
of $\E$. (Numerical examination of the analytic function
reveals that it does not have any roots at all for 
relevant values of $\E$,
making the ``all but finitely many'' caveat moot in practice.)
\end{enumerate}

We establish some notation. For any graphon $g$, we 
let $\E(g)$ and $\T(g)$ denote the densities
$t_H(g)$ where $H$ is an edge and a triangle, respectively.
For our $B(1,1)$ graphon, we let
\begin{equation} a = p_{11}, \qquad d=p_{12}, \qquad b= p_{22},
\end{equation}
and let $c$ be the size of the first pode. For a $B(2,1)$ graphon, we
assume that the first two podes are interchangable, each of size
$c/2$, and set $a_+ = p_{12}$, $a_- = p_{11} = p_{22}$,
$d=p_{13}=p_{23}$, and $b=p_{33}$. That is, the $B(2,1)$ graphon is
obtained from a $B(1,1)$ graphon by splitting the first pode in half (and
renumbering the last pode), and by making $p_{11}$ and $p_{12}$
distinct variables. The only way that a $B(1,1)$ graphon can be a limit
of $B(2,1)$ graphons is if $a_+ - a_-$ goes to zero (note that if $c\to 1$ it becomes a type $A(2,0)$ graphon,
not a $B(1,1)$). So to prove our
theorem, we must show that a sequence of $B(2,1)$ entropy maximizers
cannot approach a limiting graphon with $a_+ = a_-$.

The Euler-Lagrange equations for maximizing entropy (see \cite{KRRS1}) say that there exist constants $\alpha$
and $\beta$ such that, for all $(x,y) \in [0,1]^2$,
\begin{equation} 
\frac{\delta S(g)}{\delta g(x,y)} = \alpha \frac{\delta \E(g)}{\delta g(x,y)} + \beta \frac{\delta \T(g)}{\delta g(x,y)},
\end{equation}
or more explicitly
\begin{equation}
\label{EL1} 
S_0'(g(x,y)) 
= \alpha + 3 \beta  \int_0^1 g(x,z) g(y,z) dz, 
\end{equation} 
where 
\begin{equation}
S(g) = \int S_0(g(x,y))\,dxdy,\ \hbox{ and }\ S_0(u)=-u\log(u) - (1-u)\log(1-u). 
\end{equation}  
For a $B(2,1)$ graphon,
the integral $\int_0^1 g(x,z) g(y,z) dz$ equals 

\begin{equation}
\begin{cases}
\frac{c}{2} (a_+^2 + a_-^2) + (1-c) d^2, &\hbox{ for }x,y < \frac{c}{2} \hbox{ or } \frac{c}{2} < x,y < c;\cr 
c a_+ a_- + (1-c)d^2,  &\hbox{ for }x < \frac{c}{2} <y <c \hbox{ or } y < \frac{c}{2} < x < c;\cr
\frac{c}2 d (a_+ + a_-) + (1-c)bd,  &\hbox{ for }x < c < y \hbox{ or } y < c < x;\cr
cd^2 + (1-c)b^2, &\hbox{ for }x,y > c.
\end{cases}
\end{equation}

The Lagrange multiplier $\beta$ equals $\partial s/\partial \T$ and is 
always positive, since we can increase both the entropy and $\T$ by linearly
interpolating between our given graphon and a constant graphon. 

Subtracting equation (\ref{EL1}) for $(x,y) = (\frac{c}{3},\frac{2c}{3})$ from equation (\ref{EL1}) for 
$(x,y)=(\frac{c}{3}, \frac{c}{3})$ gives:
\begin{equation}
S_0'(a_-) - S_0'(a_+) = \frac{3\beta c}{2} (a_+ - a_-)^2.
\end{equation}
By the mean value theorem, the left hand side equals $-S_0''(u) (a_+
-a_-)$ for some $u$ between $a_-$ and $a_+$, and so
\begin{equation}
\beta = \frac{-2 S_0''(u)}{3 (a_+ - a_-)}.
\end{equation}
Since $\beta>0$ and 
$S_0''(u)=\frac{-1}{u(1-u)}$ is negative and bounded away from zero, 
$a_+$ must be greater than $a_-$ and $\beta$
must diverge to $+ \infty$  as $a_+ - a_-$ approaches zero. To
compensate, $\alpha$ must diverge to $-\infty$. This completes step 1.
  
To understand the effect of divergent Lagrange multipliers, we must
consider all of the variational equations for bipodal graphons,
including those related to changes in $c$. The edge, triangle and
entropy densities are:
\begin{eqnarray}
\E(g) & = & c^2 a + 2c(1-c) d + (1-c)^2 b \cr 
\T(g) & = & c^3 a^3 + 3c^2(1-c) ad^2 + 3c(1-c)^2 bd^2 + (1-c)^3 b^3 \cr 
S(g) & = & c^2 S_0(a) + 2c(1-c) S_0(d) + (1-c)^2 S_0(d). 
\end{eqnarray}
Taking gradients with respect to the four parameters $(a,b,c,d)$ and
setting $\nabla S = \alpha \nabla \E + \beta \nabla \T$ gives the four
equations:
\begin{eqnarray}
S_0'(a) & = & \alpha + 3 \beta (c a^2 + (1-c) d^2), \cr 
S_0'(b) & = & \alpha + 3 \beta (c d^2 + (1-c)b^2), \cr 
2 c S_0(a) \!+\! 2(1\!-\!2c) S_0(d) \!-\! 2(1\!-\!c)S_0(d) & = & \alpha(2ca + 2(1-2c)d -2(1-c)b) \cr 
+\ 3\beta (c^2 a^3 \!+\! (2c\!-\!3c^2)ad^2 \!&+&\! (3c^2\!-\!4c+1) bd^2 \!-\! (1\!-\!c)^2 b^3 ), \cr 
S_0'(d) & = & \alpha + 3\beta (cad + (1-c)bd).
\end{eqnarray}

Note that the left hand side of the third equation is always finite, and that
the left hand sides of the other equations only diverge if the 
relevant parameter
$a$, $b$, or $d$ approaches 0 or 1. Otherwise, in the $\beta \to \infty$ limit
the ratios of the coefficients
of $\beta$ and $\alpha$ must be the same for all equations. In other words,
there is a constant $\lambda = -\beta/\alpha$ such that 
\begin{equation} \nabla \T = \lambda \nabla \E,
\end{equation}
where we restrict attention to parameters that are not 0 or 1, and such
that a parameter $i\in\{a,b,d\}$ equals 0 if 
$\partial_i \T > \lambda \partial_i \E$
and equals 1 if $\partial_i \T < \lambda \partial_i \E$. (The signs 
of the inequalities come
from the sign of $S_0'(u)$ as $u \to 0$ or $u \to 1$.)

However, these are precisely the same equations that describe 
finding a stationary point for $\T$ for fixed $\E$, without regard to the 
entropy. For $B(1,1)$ graphons, such stationary points occur only for
Erd\H{o}s-R\'{e}nyi graphons, with $a=b=d$, or for minimizers of $\T$, 
with $d=1$ and $b=0$ and $c$ satisfying the algebraic condition
\be \label{min-t-eq}
0 = a^3c^2 + ac^2 + 2a^2c + 2(1-c) -4a^2c^2 -4c(1-c). \ee
This completes step 2. 

Finally, we must show that a $B(1,1)$ graphon with $b=0$ and $d=1$, and 
satisfying (\ref{min-t-eq}), is not an
entropy maximizer. Among bipodal graphons with $b=0$ and $d=1$ and a 
fixed value of $\E$, minimizing $\T$ and minimizing the entropy $s$ give
different analytic equations for $a$ and $c$. For all but finitely many values
of $\E$, these equations have distinct roots, implying that the graphon that
minimizes $\T$ does not maximize the entropy. If we start at the bipodal
graphon that minimizes $\T$ for fixed $\E$ and change $c$ to 
$c + \delta_c$,
while adjusting $a$ to keep $\E$ fixed, we can increase the entropy to
first order in $\delta_c$ while only increasing $\T$ to second order in 
$\delta_c$. 

To compensate for this increase in $\T$, we can split the first pode
in half, yielding a $B(2,1)$ tripodal graphon with
$d=1$ and $b=0$, with $a_+= a+\delta_a$ and $a_-=a-\delta_a$. 
This decreases $\T$ by $c^3\delta_a^3$, while
decreasing the entropy by $O(\delta_a^2)$. By taking $\delta_a$
to be of order $\delta_c^{2/3}$, we can restore the initial value of $\T$ at
an entropy cost of $O(\delta_c^{4/3})$.

For $\delta_c$ sufficiently small and of the correct sign, 
the $O(\delta_c^1)$ gain in entropy from
changing $c$ and $a$ is greater than the $O(\delta_a^2)=O(\delta_c^{4/3})$ 
cost in entropy
from having $a_+ \ne a_-$, so the resulting $B(2,1)$ graphon has higher entropy,
but the same values of $\E$ and $\T$, than the $B(1,1)$ graphon that minimizes $\T$ (among $B(1,1)$ graphons). This completes step 3.
\hfill$\square$

\noindent{\bf Proof of Theorem~\ref{thm:no-continuous-theorem-2}.} 
The proof follows the same strategy as that of the $m=1$ case.  In
step 1 we show that the Lagrange multipliers must diverge at a
continuous transition. In step 2 we show that divergent Lagrange
multipliers force a $B(m,1)$ graphon to be a stationary point of $\T$
for fixed $\E$. In step 3 we show that we can perturb such a
stationary $B(m,1)$ graphon into a $C(m,2)$ graphon with the same
values of $(\E,\T)$ and more entropy, implying that we are not
actually at the phase boundary, which is a contradiction.

We parametrize $C(m,2)$ graphons
as follows: There are two interchangable podes, each of size $c/2$, and $m$ podes of size $(1-c)/m$. 
Let $p_{ij}$ denote the value of $g(x,y)$ when $x$ is in the $i$-th pode and $y$ is in the $j$-th. We define parameters
$\{a_+, a_-, b, d, p\}$ such that
\begin{equation} p_{ij} = \begin{cases} a_+ &  \hbox{if }(i,j)=(1,2) \hbox{ or }(2,1), \cr
a_- &  \hbox{if }(i,j)=(1,1) \hbox{ or }(2,2), \cr
d &  \hbox{if }i \le 2 <j \hbox{ or } j \le 2 < i, \cr
b & \hbox{if } i=j>2, \cr
p & \hbox{if } i \ne j \hbox{ and } i,j>2.
\end{cases}
\end{equation}

The transition that we are trying to rule out is one where $a_\pm$ approach a common value $a$, resulting in a
$B(m,1)$ graphon with one pode of size $c$ and $m$ podes of size $(1-c)/m$, with
\begin{equation} p_{ij} = \begin{cases} a &  \hbox{if }(i,j)=(1,1), \cr
d &  \hbox{if }i =1  <j \hbox{ or } j =1 < i, \cr
b & \hbox{if } i=j>1, \cr
p & \hbox{if } i \ne j \hbox{ and } i,j>1.
\end{cases}
\end{equation}

Let $\nabla S$, $\nabla \E$ and $\nabla \T$ be the gradients of the functionals $S$, $\E$ and $\T$ with respect to the parameters
$(a,b,c,d,p)$ or $(a_+,a_-,b,c,d,p)$, depending on the phase we are considering. Maximizing the entropy for fixed $(\E,\T)$ means finding Lagrange multipliers $\alpha$ and $\beta$ such that 
\begin{equation}\label{EL2s} \nabla S = \alpha \nabla \E + \beta \nabla \T. \end{equation}
In the $C(m,2)$ phase, one checks that $\partial S/\partial a_- - \partial S/\partial a_+$ is first order in $(a_+-a_-)$, while 
$\partial \T/\partial a_- - \partial \T/\partial a_+$ is second order. This forces $\beta$ to diverge to $+\infty$ as 
$a_+ - a_- \to 0^+$, 
which in turn forces $\alpha$ to diverge to $-\infty$, exactly as in the proof of Theorem
\ref{thm:no-continuous-theorem-1}. This concludes step 1.

Next we consider equation (\ref{EL2s}) in the limit of divergent $\alpha$ and $\beta$. Restricting attention to
those parameters for which $\nabla S$ does not diverge (i.e. those that are not 0 or 1 in the limiting graphon), we
have that $\nabla \E$ and $\nabla \T$ must be collinear. That is, there is a constant $\lambda$ such that,
for each parameter $q \in (a,b,d,p)$ taking
values in $(0,1)$, we must have that $\partial \T /\partial q = \lambda \partial \E/\partial q$. 
Furthermore, if in the limit $q$ goes to 0 we must have $\partial \T/\partial q \ge \lambda \partial \E/\partial q$ and if 
$q$ approaches 1 we must have $\partial \T/\partial q \le \lambda \partial \E/\partial q$.  That is,
\begin{equation} \label{EL3s} \nabla \T = \lambda \nabla
  \E, \end{equation} where we restrict attention to parameters that
are not 0 or 1, and we have 1-sided inequalities for those remaining
parameters.  This is {\em precisely} the set of equations obtained by
ignoring entropy and looking for stationary points of $\T$ for fixed
$\E$, with $\lambda$ being the Lagrange multiplier of this process.
Seeking entropy maximizers with divergent $(\alpha, \beta)$ is
equivalent to seeking stationary points of $\T$ for fixed $\T$. This
concludes step 2.

Now we consider a 1-parameter family of graphons, with a fixed value of $\E$, that satisfy all of (\ref{EL3s}) 
except the equation relating $\partial \T/\partial c$ and $\partial \E/\partial c$. Since by assumption we start at 
a stationary point of $\T$, moving along this family will only
change $\T$ to second order or slower in the change in $c$, but for all but finitely many values of $\E$ will change $S$ to 
first order. Move a distance $\delta_c$ in the direction of increasing $S$. A priori we do not know that the resulting
change in $\T$ will be positive, but negative changes in $\T$ can be compensated for while {\em increasing} $S$, since $\beta$ is positive.
If the change in $\T$ is positive, then we can compensate by splitting the first pode in half, with $a_+-a_-$ of 
order $(\delta_c^{2/3})$, restoring $\T$ at an entropy cost of $O(\delta_c^{4/3})$. 
Since $O(\delta_c^1) > O(\delta_c^{4/3})$, by picking $\delta_c$ small enough we can
always find a $C(m,2)$ graphon that does better than the $B(m,1)$ graphon that was purportedly the entropy
maximizer at the phase boundary, which is a contradiction.  
\hfill$\square$

\noindent{\bf Proof of Theorem \ref{thm:3}.}
This proof does not require any consideration of entropy. 
The symmetries are simply incompatible, 
in that there is no way to approximate a
$B(m-1,1)$ graphon with a $B(m,1)$ graphon, and there is no way to approximate
a $B(n-2,1)$ graphon with an $A(n,0)$ graphon.  
\hfill$\square$
 
We now turn to determining the locations of the phase transitions. For each of the discontinuous transitions, this is a difficult 
problem. The optimizing graphons on each side of the transition line are very different, so it is impossible to use perturbation theory to understand the behavior near the line. Instead, one must study each phase separately and approximate the entropy in
each phase as an analytic funciton of $\E$ and $\T$ (e.g., by doing a polynomial fit to numerical data). These functions can then
be continued over a larger region and compared. The phase transition line is the locus where the two functions are equal. Using 
such techniques, we can localize the transition lines and the triple points with considerable accuracy, but in the end the results 
remain grounded in numerical simulation, and cannot provide independent confirmation of our numerics. 

For the continuous $A(n,0) \leftrightarrow B(n-1,1)$ transitions, however, it is possible to obtain an analytic equation satisfied
along the transition line. This was already done for the $A(2,0) \leftrightarrow B(1,1)$ transition in \cite{RRS1,RRS2}. Here we extend
the results to $n>2$. 

This calculation, although elementary, is too long to present here in
its entirety; we give the method here.  For each fixed $n>2$, the
space of $B(n-1,1)$ graphons is 5-dimensional. As in Figure 3, we
imagine $n-1$ intervals of size $\frac{c}{n-1}$ and one of size $1-c$,
and must specify $a=p_{1,1}$, $b=p_{1,2}$, $d=p_{1,n}$, $p=p_{n,n}$
and $c$. An $A(n,0)$ graphon is a special case of this with $p=a$,
$d=b$ and $c=\frac{n-1}{n}$, and for such graphons the parameters $a$
and $b$ are easily computed from the edge and triangle densities:
\begin{equation}\label{a-and-b}
a = \E - (n-1) \left(\frac{\E^3 - \T}{n-1}\right )^{1/3}; \qquad  
b = \E +  \left(\frac{\E^3 - \T}{n-1}\right )^{1/3}.
\end{equation} 
Returning to a general $B(n-1,1)$ graphon, we use the constraints on
$\E$ and $\T$ to eliminate two variables, expressing $d,p$  as functions of $a,b,c$: $d=d(a,b,c), p=p(a,b,c)$.
In fact we don't need to solve explicitly; we only need the first and second partial derivatives of $d(a,b,c),p(a,b,c)$
(evaluated at the parameter values of the $A(n,0)$ graphon), which can be obtained by implicit differentiation of the equations for $\E,\T$. 
Then the entropy $S$ is a function of $a,b,c$: $S=S(a,b,c,d(a,b,c),p(a,b,c)).$ Computing its Hessian
$H(S)$ at the $A(n,0)$ phase when $d=b,p=a$ yields
the matrix
\begin{equation}
H(S)=\begin{pmatrix}
 \frac{(n-1) \left((a-b)^2-4 (a-1) a^2 X\right)}{(a-1) a (a-b)^2 n} & -\frac{2 b (n-2) (n-1) X}{(a-b)^2 n} & -\frac{2 (a+b) X}{a-b} \\
 -\frac{2 b (n-2) (n-1) X}{(a-b)^2 n} & \frac{(n-2) (n-1) \left((a-b)^2-2 (b-1) b (2 a+b (n-4)) X\right)}{2 (a-b)^2 (b-1) b n} & -\frac{2 b (n-2) X}{a-b} \\
 -\frac{2 (a+b) X}{a-b} & -\frac{2 b (n-2) X}{a-b} & \frac{2 n (\log (1-b)-\log (1-a))}{n-1}
\end{pmatrix}
\end{equation}
where $X=\tanh^{-1}(a)-\tanh^{-1}(b)$.  Within the $A(n,0)$ phase,
this second variation is negative-definite, while within the
$B(n-1,1)$ phase the matrix has a positive eigenvalue. The boundary is
thus defined (locally) by the analytic equation $\det H(S)=0$.

One can also consider continuous changes from an $A(n,0)$ graphon to a 
graphon with a symmetry other than $B(n-1,1)$. The local stability condition 
for such changes works out to be exactly the same as for 
$A(n,0) \leftrightarrow B(n-1,1)$. The upshot is that the analytic curve
$\det H(S)=0$ defines the boundary of the region where the $A(n,0)$ graphon
is stable against small perturbations. 

In Figure \ref{FIG:Stability-regions} we plot the regions where the 
$A(3,0)$ and $A(4,0)$ graphons are stable against small perturbations.
Comparing to Figures 6 and 7, we see that the stability regions are larger
than the actual $A(3,0)$ and $A(4,0)$ phases, and even overlap! 
Without assuming {\em anything} about the phase portrait (beyond the
existence of $A(3,0)$ and $A(4,0)$ phases) this proves that 
some of the transitions from $A(3,0)$ or $A(4,0)$ to other phases must not
be governed by the local stability condition, and so must be discontinuous. 

\begin{figure}[!htb]
\centering
\includegraphics[width=0.45\textwidth,height=0.41\textwidth]{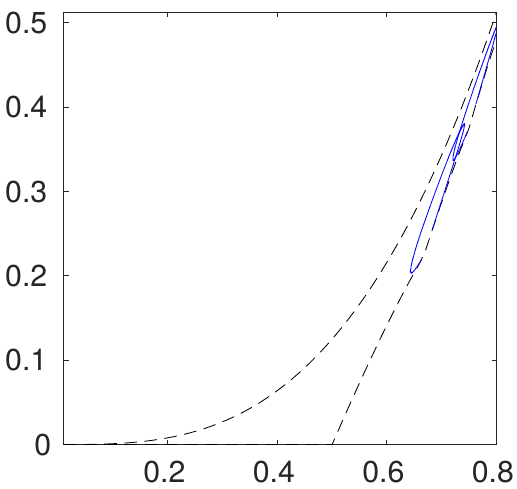} \hskip 1cm
\includegraphics[width=0.45\textwidth,height=0.41\textwidth]{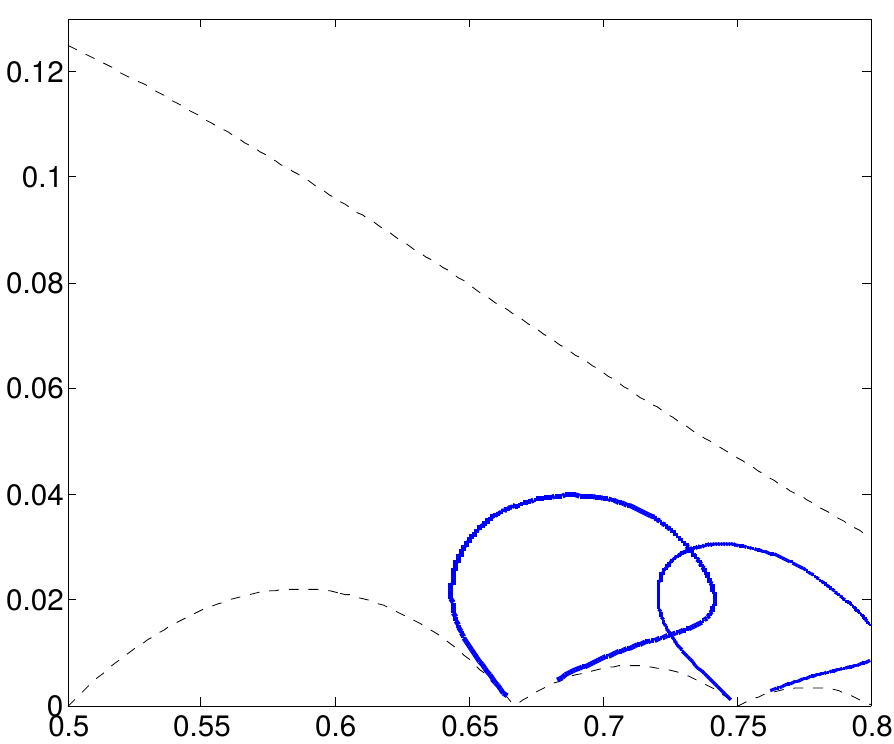}
\caption{The regions of local stability for $A(3,0)$ and $A(4,0)$
  graphons, showing overlap.
The left plot uses coordinates $(\E,\T)$ and the right plot uses
$(\E, \T'=\T-\E(2\E-1))$.}
\label{FIG:Stability-regions}
\end{figure}

The transition $A(2,0) \leftrightarrow B(1,1)$ is a supercritical pitchfork
bifurcation \cite{RRS2}. The analytic family of entropy maximizers in the 
$A(2,0)$ phase has a natural analytic continuation into the $B(1,1)$ region, 
but no longer maximizes entropy there. The analytic family of entropy maximizers
in the $B(1,1)$ phase doubles back on itself and cannot be continued into
the $A(2,0)$ region. Power series analysis suggests 
that something similar happens in 
the transition from $A(n,0)$ to $B(n-1,1)$. The family of graphons representing
$A(n,0)$ entropy maximizers can be continued into the $B(n-1,1)$ region but
no longer maximizes entropy. The family of graphons representing $B(n-1,1)$ 
cannot be analytically continued into the $A(n,0)$ region, but instead 
doubles back into the $B(n-1,1)$ region. Unlike in the case of $B(1,1)$, the
``doubling back'' branch is not related by symmetry to the original branch,
and so represents a different set of reduced graphons, all stationary
points of the entropy, but with presumably lower entropy than graphons of the
original branch with the same values of $(\E,\T)$.   

It is a classical result (Mantel's theorem~\cite{Ma}, generalized by
Tur\' an~\cite{Tu}) that the only way to satisfy the constraints of
edge density $\E=1/2$ and triangle density $\T=0$ is with the complete,
balanced bipartite graph, which implies that those values of those
density constraints determine the values of the densities of all other
subgraphs. Put another way, there is a {\em unique} reduced graphon $g$ with
$\T(g)=0$ and $\E(g)=1/2$.  Likewise, for each $\E$ there is a unique
reduced graphon with the maximum possible value of $\T=\E^{3/2}$. 
This phenomenon is called `finite forcing'; see~
\cite{LS}. However, this phenomenon only occurs on the boundary of the
phase space:

\begin{theorem}For each pair $(\E_0,\T_0)$ in the interior of the space of
achievable values (see Figure~\ref{Razborov-triangle}), there exist
multiple inequivalent graphons $g$ with $\E(g)=\E_0$ and $\T(g)=\T_0$. 
\end{theorem}

\begin{proof}
For fixed $\E_0$, let $g_0$ be a graphon that minimizes $\T(g)$ given
$\E(g)=\E_0$, and let $g_1$ be a graphon that maximizes $\T(g)$, and 
let $\phi: [0,1] \to [0,1]$ be a general measure-preserving homeomorphism. 
The graphon $g_0$ is always $n$-podal for some $n$, while $g_1$ is 
bipodal. For $t \in (0,1)$, let $g_{t,\phi} = t g_0 + (1-t) g_1 \circ \phi$. 
This will be a multipodal graphon, generically with $2n$ podes whose sizes 
depend on the details of $\phi$ but not on the value of $t$. In particular,
we can choose homeomorphisms $\phi_1$ and $\phi_2$ such that the podal 
structure of $g_{t,\phi_1}$ is different from that of $g_{t,\phi_2}$, and hence
is different from that of $g_{t',\phi_2}$ for arbitrary $t' \in (0,1)$.

It is easy to check that $\E(g_{t,\phi}) = t \E(g_0) + (1-t) \E(g_1\circ \phi)
= \E_0$, and that for given $\phi$, $\T(g_{t,\phi})$ is a continuous function 
of $t$. By the intermediate
value theorem, we can thus find graphons $g_{t_1,\phi_1}$ 
and $g_{t_2,\phi_2}$ such that 
$\E(g_{t_1,\phi_1})=\E(g_{t_2,\phi_2})=\E_0$ and 
$\T(g_{t_1,\phi_1})=\T(g_{t_2,\phi_2})=\T_0$. But 
$g_{t_1,\phi_1}$ and $g_{t_2,\phi_2}$ have different podal structures, and so 
are inequivalent. 
\end{proof}

In contrast to the graphons not being determined by $\E$ and $\T$,
simulations indicate that maximizing the entropy for fixed $(\E,\T)$
does give a unique reduced graphon {\it throughout the interior,
  except  on a lower dimensional set}, the exceptions being
constraint values associated with the discontinuous phase transitions.
(The uniqueness of the entropy maximizer was also proven analytically
for $(\E,\T)=(1/2,\T)$ for any $0\le \T\le 1/8$ in~\cite{RS1}, and for
an open subset of the $F(1,1)$ phase in~\cite{KRRS2}.)  We do not yet
have a theoretical understanding of this fundamental issue, sometimes
called the Gibbs phase rule in physics; see however~\cite{KRRS1} for
$k$-star graph models, and~\cite{Ru2,I} for weak versions in physics.

\old{In some sense this uniqueness is in the spirit of finite forcing, in
that the values of a finite set of functionals $\{\E, \T, S \}$ picks
out a single graphon from an infinite-dimensional family. However,
this only seems to occur on the boundary of the space of achievable
values of $(\E, \T, S)$, specifically the constrained maximum value of
$S$.}

\section{Symmetry}
\label{SEC:symmetry}

All aspects of this paper relate to the symmetry of phases, referring
to the symmetry of the unique entropy-optimizing graphons $g_{\E,\T}$
for points $(\E,\T)$ in the phase. These symmetries occur at two
levels. 

The first and most significant level of symmetry is a consequence of the multipodality of
the $g_{\E,\T}$, which means that the set of all nodes is composed of
a finite number of equivalence classes: the probability of an edge,
between a node $v_j$ in class $j$ with a node $v_k$ in class $k$, is
independent of $v_j$ and $v_k$, only depending on $j$ and $k$. 

The second level
of symmetry concerns the equivalence classes of nodes:
certain equivalence classes have the same sizes and edge probabilities, and others don't.  
These size and probability parameters are used to distinguish distinct phases, that is,
maximal open regions in the parameter space where the entropy-maximizing graphon $g_{\E,\T}$ is unique and varies analytically.

Because of multipodality the function $g_{\E,\T}$ restricted to a given phase can be
considered a smooth vector valued function of $(\E,\T)$, the
coordinates being the probabilities of edges between node equivalence
classes, and the relative sizes of those equivalence classes.
By the symmetry of a phase we refer to the symmetries among these
coordinates, with the following caveat, illustrated through an example. Within the phase
$F(1,1)$ there is a curve such that the two node equivalence classes
have the same size. In a narrow sense this might have signalled a
higher symmetry. However there is no singular behavior as $(\E,\T)$
crosses this curve so the curve is simply part of $F(1,1)$ and does
not affect the `symmetry' of the phase. 

Each of the phases has a different symmetry and we
conjecture that they all fall into the three families: $A(n,0)$,
$B(n,1)$ and $C(n,2)$,  or $F(1,1)$, in which the notation completely specifies the
symmetry (except for the pairs $F(1,1),\ B(1,1)$, and $C(1,2),\ B(2,1)$).

An important result of this paper is the conjectured phase diagram,
Figure~\ref{phase-diagram}. The other goal is to show how knowledge of the
structure of optimizing graphons in phases can be helpful in
understanding the role of symmetry in specific features of phase
transitions. 
Consider the following, paraphrased from P.W.~Anderson~\cite{An}, a picture he attributes to Landau~\cite{LL}:

\leftskip=0.5truein
\rightskip=0.5truein
\noindent
  The First Theorem of solid-state physics states that it is
  impossible to change symmetry gradually. A symmetry element is
  either there or it is not; there is no way for it to grow
  imperceptibly.

\leftskip=0truein
\rightskip=0truein

\noindent
This intuitive picture has been applied, for instance by Landau, to
understand why there is no critical point for the fluid/solid
transition~\cite{An}, though it has been difficult to make the
argument rigorous:

\leftskip=0.5truein
\rightskip=0.5truein
\noindent
  This is the theoretical argument, which has appeared to some to be a
  little too straightforward to be absolutely convincing~\cite{Pi}.

\leftskip=0truein
\rightskip=0truein

\noindent
We suggest that network models
such as the edge/triangle model of this paper provide a useful
framework for enabling a rigorous study of such symmetry
principles. This was done in~\cite{RRS2} specifically for the issue of
existence of a critical point. 

Landau's symmetry principles are commonly applied to the issue of
whether phases are continuous at a transition~\cite{LL},
which is also related to uniqueness of entropy optimizers. 
 
We have proven that in the edge/triangle model certain transitions cannot be
continuous, and evidence suggests that certain other transitions are
continuous. 
It is worthwhile discussing how these transitions are approached at the micro-level, that
is, in terms of the multipodal parameters, the probabilities of edges
between various type of nodes. In this regard Figure~\ref{phase-diagram}
and Figure~\ref{FIG:Line E0p735} are useful.

For all the discontinuous transitions, those proven and those only
seen in simulation, we believe from simulation that the transitions can
be visualized as the intersection of a pair of two dimensional smooth
surfaces both of which exist beyond the intersection but only
represent entropy-optimizers on one side. See rows $(i),
(iv),(v),(vi)$ in Figure~\ref{FIG:Line E0p735}.

For the continuous transitions there is more variety. By simulation
the continuous $B(2,1)\leftrightarrow A(3,0)$ transition is achieved
through directly acquiring the higher symmetry of $A(3,0)$. See rows
$(ii)$ and $(iii)$ in Figure~\ref{FIG:Line E0p735}. An analogue in
statistical mechanics would be a transition between crystal phases in
which a rhombohedral unit cell becomes, and remains, cubic. On the other hand the
$A(2,0)\leftrightarrow F(1,1)$ transition occurs by the different 
bipodal symmetries on both
sides rising to the full symmetry of the constant graphons at the
Erd\H{o}s-R\'{e}nyi curve. It is noteworthy that the full symmetry of
the constant graphon is incompatible with any possible phase in our sense: since $\E=\T^3$, there
is not a two-dimensional family of parameter values.
\section{Conclusion}
\label{SEC:conclusion}
The edge/triangle model in this paper was built by
analogy with microcanonical mean-field models in statistical
mechanics. Mean-field models are useful because frequently the free
energy can be determined analytically as a function of the
thermodynamic parameters.  An important distinction for the
edge/triangle and related random graph models is that not only the free energy
(entropy in this case) but also the entropy-optimizing graphons (which are the
analogues of the Gibbs states) can sometimes be determined for a range of
parameters. For instance in~\cite[Section 3.7]{RRS2}
this control of the optimizing states is used to compute how some
global quantity changes with the constraint parameters, a level of
analysis never possible in short or mean-field models in
statistical mechanics.

The main result of this paper is the conjectured phase
diagram, Figure~\ref{phase-diagram}, for edge/triangle constraints,
based largely on simulation, including continuity/discontinuity of all
the transitions and the structure of the entropy-optimizing states
within each phase.

The secondary goal is to show how knowledge of the structure of
the optimizing graphons can be helpful in understanding the role of
symmetry in the continuity/discontinuity of phase transitions.

Interesting subjects for further investigation include the triple points where 
phases $F(1,1)$, $A(2,0)$ and $B(1,1)$ meet, all the pairwise transitions
being continuous, and where $B(1,1)$, $A(3,0)$
and $B(2,1)$ meet, with $B(1,1)\leftrightarrow A(3,0)$ and $B(2,1)
\leftrightarrow B(1,1)$ discontinuous and $A(3,0)\leftrightarrow
B(2,1)$ continuous.

\section*{Acknowledgments}

The main computational results were obtained on the
computational facilities in the Texas Super Computing Center
(TACC). We gratefully acknowledge this computational support.
This work was also partially supported by NSF grants DMS-1208191, DMS-1612668,
DMS-1509088, DMS-1321018 and DMS-1620473, and Simons Investigator grant  327929.



\end{document}